\documentclass[reqno]{amsart}
\usepackage[utf8]{inputenc}

\usepackage{comment}
\usepackage{bbm}
\usepackage[foot]{amsaddr}

\usepackage{geometry}
\usepackage{graphicx}
\graphicspath{{Figures/}}
\usepackage[usenames,dvipsnames]{xcolor}
\usepackage{dsfont,url}
\usepackage[american]{babel}

\newtheorem{theorem}{Theorem}

\newtheorem{assumption}{Assumption}

\newtheorem{definition}[theorem]{Definition}

\newtheorem{lemma}[theorem]{Lemma}

\newtheorem{remark}{Remark}

\def\qed{\hbox{${\vcenter{\vbox{		 
   \hrule height 0.4pt\hbox{\vrule width 0.5pt height 6pt
   \kern5pt\vrule width 0.5pt}\hrule height 0.4pt}}}$}}

\def\cB{\mathcal B}

\def\cF{\mathcal F}

\def\cH{\mathcal H}

\def\bE{\mathbb E}

\def\bP{\mathbb P}

\def\bR{\mathbb R}

\def\eps{\varepsilon}

\def\ER{Erd\H{o}s-R\'enyi }

\def\diag{\mathop{\rm diag}}

\newcommand{\so}[1]{\textcolor{black}{#1}}
\newcommand{\ste}[1]{\textcolor{black}{#1}}

\newcommand{\rev}[1]{\textcolor{black}{#1}}

\begin{document}

\title{A stochastic differential equation SIS model on network under Markovian switching}
\author{Stefano Bonaccorsi, Stefania Ottaviano}

\address{Stefano Bonaccorsi, \newline
Dipartimento di Matematica, Universit\`a degli studi di
   Trento, via Sommarive 14, 38123 Povo (Trento), Italy
}

\address{Stefania Ottaviano, \newline
Dipartimento di Matematica “Tullio Levi Civita”, Universit\`a degli studi di
    Padova, Viale Trieste 63, 35131 Padova\\
Dipartimento di Matematica, Universit\`a degli studi di
   Trento, via Sommarive 14, 38123 Povo (Trento), Italy
}
\email{stefano.bonaccorsi@unitn.it}
\email{stefania.ottaviano@unipd.it}


\markboth{living document,~\today}%
{S.Bonaccorsi}


\begin{abstract}
We study a stochastic SIS (susceptible-infected-susceptible) epidemic dynamics on network, under the effect of a Markovian regime-switching. We first prove the existence of a unique global positive solution, and find a positive invariant set for the system. 
Then, we find sufficient conditions for a.s. extinction and stochastic permanence, showing also their relation with the stationary probability distribution of the Markov chain that governs the switching \rev{and with the network topology}. We provide an asymptotic lower bound for the time average of
the 
{sample-path solution}
under the conditions ensuring stochastic permanence. From this bound, we are able to prove the existence of an invariant probability measure if the condition of stochastic permanence holds. \ste{Under a different condition, we prove the positive recurrence and the ergodicity of the regime-switching diffusion.}
\end{abstract}

\maketitle

\section{Introduction}

Epidemic processes are often subject to environmental noise. It is therefore useful to understand how the noise influences the epidemic systems.
To this aim, it is interesting to consider randomly switching dynamical systems, since in the
real biological systems some parameters of the model are usually influenced by random variation, 
thus they are not constant in time. For example, the disease transmission rate in some epidemic models is influenced by external meteorological factors linked to the survival of many bacteria and viruses \cite{arundel1986indirect,ud2010structural}.

Moreover, also the possible routes of disease transmission among the population may evolve in time.
In large populations each individual
only interacts with a few others, and these connections determine the possible dynamics of the epidemics. 
\rev{Although one  common assumption in other population models is the homogeneity of interactions between individuals, 
we may object that homogeneous mixing ignores important social structures, such as presence of communities, or the specific role of the individuals, often leading to inaccurate reconstructions and predictions of the route of epidemics.
Therefore, the use of networks to describe the contact pattern represents a major advance in our
ability to model realistic social behaviours. 
Network-based models have succeeded in explaining characteristic
patterns observed in recent real epidemics (e.g., in the 2009 H1N1 influenza, the 2014-2016 Ebola epidemic in West
Africa and the COVID-2019 pandemic) that classical homogeneous epidemic models are not able
to catch, such as an early sub-exponential growth of new cases, cluster transmission and superspreading
events} \cite{balcan2009seasonal,ajelli20152014,faye2015chains,who2016after,thurner2020network,gatto2020spread}.\\
In network-based models, individuals are represented by
nodes, with edges (or links) depicting the interactions between them. 
\rev{However, complex and heterogeneous connectivity patterns emerge also in a
wide range of biological and other socio-technical systems: biological
system are the result of biochemical reactions; network structure can be recognizable
even in the Internet, or in an electric power grid, in the physical
layer of the telecommunication systems, in highways and subways systems or
neural networks \cite{boccaletti2006,ottaviano2016influence}}.
\rev{We can also see the nodes as municipalities/metropolitan areas, country regions, airports, or anyway like any homogeneous population that shares some common property (as, e.g., age), and consider, possibly, edges equipped with a weight 
as indicators of the transmission chance between nodes.
However, in many realistic scenarios, links and, if considered, weights should be treated as dynamical variables possibly affected by random disturbance.}
Thus, the network structure, described for instance by the connectivity matrix, varies in time, and its variation shall be modelled by a stochastic process;
switching on and off some edges can be made according to a suitable
random law that leads to a change of the epidemic dynamics as time goes on.

In epidemiology, many authors have considered random switching systems, whose distinctive feature is the coexistence of a continuous dynamics and discrete events (random jumps
at points in time). 
In \cite{gray2012sis}, the authors discuss the effect of a random switching environment on the SIS epidemic
model using a finite state Markov chain. Other kinds of epidemic models under Markovian regime switching are studied in \cite{greenhalgh2016modelling,li2017threshold}. Nevertheless, aforementioned
epidemic models are based on the traditional homogeneous compartment models which do not take into
account the topology of social contact networks.
 With regard to the spread of the epidemic in networks under regime switching in \cite{prakash2010virus} and \cite{sanatkar2015epidemic}, the authors study the switching
between networks in discrete-time and continuous-time respectively, getting an epidemic
threshold. The networks switching they used is deterministic and periodic. A random switching mechanism is considered, e.g., in \cite{ogura2015disease,cao2019n}, where the authors investigate the so-called switched $N$-intertwined mean-field model. Basically, they consider a piece-wise deterministic Markov process, where the deterministic continuous part is governed by the $N$-intertwined mean-field equation \cite{van2009virus}.\\
In recent years, there has been also a growing interest in hybrid diffusion systems (often called regime-switching diffusions) \cite{zhu2007asymptotic}, where 
the continuous component between jumps follows a stochastic differential equation. Thus, we have a diffusion process living in a random environment (depicted by the finite-state jumping process).
Regime-switching diffusions
provide a more realistic description for many fields of application, including epidemiology and population dynamics \cite{han2013stochastic,zhang2016stationary,jin2018asymptotic, lin2019ergodicity,wang2019dynamical,lan2019stochastic,li2009population,du2016conditions, el2021stochastic,wei2021bifurcation}. \rev{Recently, the long-term behavior of a regime-switching diffusion SIS epidemic model, obtained by combining the model in \cite{cai2019stochastic} with that in \cite{gray2012sis}, has been investigated in \cite{cai2021stochastic}.
A class of SIS hybrid diffusion models is studied also in \cite{tran2021optimal}, where the authors focus on an optimal control problem. }

 However, to the best of our knowledge, hybrid diffusion epidemic models that consider a network structured population are still missing. Thus, our work is the first attempt to fill this gap.
\rev{ Precisely, we will consider an SIS epidemic model on networks, where the continuous part, accounting for the evolution of the infection probability of each node, evolves according to a stochastic differential equations, while the regime-switching is governed by a finite-state Markov chain}.

\rev{In the SIS model a node can be repeatedly infected, recover, and yet be infected again. Thus, it can be used to describe diseases that do not confer permanent immunity, such as 
some sexually transmitted and bacterial
diseases as gonhorrea, tubercolosis, and Streptococcus pneumoniae 
\cite{lajmanovich1976deterministic, Gray2011,gray2012sis}. We underline that the SIS model can be used also for describing diffusion of some computer viruses \cite{kephart1992directed,van2009virus}. \\
Moreover, the regime-switching approach is well suited to represent meteorological factors that influence the survival of certain bacteria and viruses, but also to represent the switching between different serotypes (or strains) leading to different infectivities and infectious periods (see, e.g., \cite[Example 8.1]{gray2012sis}).
 The regime-switching in our model is suited also for
the change in the route of transmission, encoded by randomly switching on and off some edges of the social contact network.
In the next subsection, we will precisely define our model and the results obtained.}


  %

\subsection{The model}
 
\rev{ Let us consider the NIMFA (N-intertwined mean-field approximation) model proposed in \cite{van2009virus} for an agent-based SIS epidemic. The model was subsequently extended to other classes of epidemic models such as SIR and SIRS in \cite{SIRscoglio} and \cite{ottaviano2021some}, respectively.\\
 In the case of the NIMFA-SIS model, the dynamics of the infection probability of each node, in a static undirected connected graph, is governed by the following system of nonlinear ODEs}:

\begin{align}\label{nimfa}
\frac{d x(t)}{dt} &=( \beta A -\delta I) x(t) - \beta \diag {(x_i(t))} A x(t),
\end{align}
where $x(t) =(x_1(t), x_2(t), . . . , x_N(t))^T$ is the vector whose component $x_i(t)$ represents
the probability that the $i$-th node is infected at time $t$. The parameters $\beta$ and $\delta$ represent the infection and recovery rates, respectively, $A$ is the adjacency matrix of the network, and $I$ the $N \times N$ identity matrix.


 First, we consider the random switching of the environment between two or more states. In each different state the parameters takes different values and the network changes its link set.
  We assume that the switching is memoryless, i.e., the waiting time for the next
switch has an exponential distribution. 

In this paper, unless otherwise specified, let $(\Omega,\mathcal{F}, \{\mathcal{F}_t\}_{ t\geq 0}, \bP)$ be a complete
probability space with a filtration $\{\mathcal{F}_t\}_{ t\geq 0}$ satisfying the usual
conditions (i.e. it is right continuous and $\cF_0$ contains all $\bP$-null
sets).\\
The infection probabilities dynamics in \eqref{nimfa} under regime switching can  be
described by the following stochastic system 

\begin{align}\label{sis_state}
\frac{d x(t)}{dt} &=( \beta(s(t)) A(s(t)) -\delta(s(t)) I) x(t) - \beta(s(t)) \diag {(x_i(t))} A(s(t))x(t),
\end{align}
 where the values of $\beta$, $\delta$ and $A$ change according to a homogeneous continuous-time Markov chain $\{ s(t), t \geq 0\}$ defined on the probability space, taking values in a finite state space $S=\{1, 2, \ldots m\}$, representing different regimes. 
The Markov chain $s(t)$ is generated by the transition rate matrix $Q=(q_{w z})_{m \times m}$, i.e., 

\begin{equation*}
\bP\{s(t+ \Delta t)= z| s(t)=w\}=\begin{cases}
q_{wz}\Delta t + o(\Delta t), \qquad \text{if} \quad w \neq z,\\
1+ q_{wz}\Delta t + o(\Delta t), \qquad \text{if} \quad w=z,
\end{cases}
\end{equation*}
where $\Delta t >0$, $q_{w z}$ is the transition rates from state $w$ to state $z$, and $q_{wz}\geq 0$ if $w \neq z$, while $q_{ww}= -\sum_{w \neq z} q_{w,z}$.
We assume that the Markov chain $s(t)$ is irreducible. Under this condition, the Markov chain has a unique stationary positive probability distribution $\pi=(\pi_1, \ldots, \pi_m)^T$ which can be determined by solving the following linear equation $\pi^T Q =0$, subject to $\sum_{s=1}^m \pi_z=1$, and $\pi_s >0$, $\forall s \in S$.\\\\


We also consider random fluctuations, that continuously affect the evolution of the process between consecutive jumps. Indeed, model parameters may have great variability, depending on observed and measured data, that are fraught with errors and uncertainty, or that simply derive from random fluctuations of the environment. To model this scenario, we consider white noise (see e.g., \cite{bonaccorsi2016epidemics,mao2002environmental,widder2014heterogeneous,li2009population}). 
 White noise is a useful mathematical idealization for representing stochastic disturbances fluctuating rapidly, which  are assumed to be uncorrelated for different instant of time \cite{Arnold}.
 
 We assume that the rate of exposure to the infection, for each individual, varies around a common average value under the action of a family of independent, identically distributed Brownian motions. In this way, we incorporate also a more realistic heterogeneity in the transmission of the infection, that may depends on many different aspects, such as genetics, immune system, or social behaviors.\\
Precisely, let us consider a standard $N$-dimensional Brownian motion $W(t) = (w_1(t), \ldots, w_N(t))$ defined on its stochastic basis $(\Omega,\mathcal{F}, \{\mathcal{F}_t\}_{ t\geq 0}, \bP)$. 
Then, we assume that each node $i$ can be infected from its neighbors by a rate described by a stochastic process of the form
\begin{align*}
 \beta(s(t)) + \sigma_i(x_i(t),s(t)) \, \dot w_i(t),
\end{align*}
 where $\dot w_i(t)$ is the white-noise mapping (see e.g., \cite[Chapter 3]{Arnold}). 
 The functions $\sigma_i: \mathbb{R} \times S \rightarrow [0, + \infty)$, {that provide the intensity of the noise
 for each node}, 
are locally Lipschitz continuous and bounded, and satisfy  
 \begin{equation}\label{cond}
\sup_{x \in (0,1)} \frac{\sigma_i(x,s)}{x} \le M(s), \quad \text{for every\ } i = 1, \dots, N, \; s \in S.
\end{equation}
 where $M(s)$ is a positive constant depending on the state of the Markov chain.
This choice implies that the intensity of the infection rate varies around a mean value $\beta(s)$, and the disturbance is small if the  value of the probability of infection is small.
Let us note that the choice of a noise whose intensity is independent of the epidemic level would have implied a random variability of the solution also near the \rev{zero-infection probability state, thus allowing the solution to have some negative components. 
In this case, the solution would enter in a region that has neither physical nor mathematical meaning}.
Moreover, {we require that $\sigma_i(x_i,s) \neq 0$, if $x_i \in (0,1)$.}\\
Thus, by perturbing system \eqref{sis_state}, we obtain the It\^o stochastic differential equation

\begin{equation}
\label{ito}
\begin{aligned}
&{\mathrm d} x_i(t) = \left[\beta(s(t)) b_i(x(t),s(t)) (1 - x_i(t)) - \delta(s(t)) x_i(t)\right] \, {\mathrm d}t
 + \sigma_i(x_i(t),s(t)) b_i(x(t),s(t)) (1 - x_i(t)) \, {\mathrm d}w_i(t),
\\
& b_i(x(t),s(t)) = \sum_{j=1}^N a_{ij}(s(t)) x_j(t), \qquad i \in \{1, \dots, N\}
\end{aligned}
\end{equation}
with a given vector of initial conditions $x_0=(x_1(0), \dots, x_N(0))$ and a given initial state $s_0=s(0)$. 
We assume that the Markov chain $s(\cdot)$ is independent of the $N$-dimensional Brownian motion $W(\cdot)$.\\ Thus, system \eqref{ito} is composed  of $m$ subsystems, and it continues to switch between them according to the law of the Markov Chain.
In the vector-valued form the above stochastic differential equation becomes
\begin{equation*}\label{sdevect}
\begin{aligned}
&{\mathrm d} x(t)=  f(x(t),s(t)) \, {\mathrm d}t + g(x(t),s(t)) \, {\mathrm d}W(t),  
\\
& (x_1(0), \dots, x_N(0),s_0),
\end{aligned}
\end{equation*}
where 
$f(x(t),s(t))$ and $g(x(t),s(t))$ are functions taking values in $\mathbb{R}^N$ and $L(\bR^N,\bR^N)$, respectively. The $i$-th component of $f$ is 

\begin{equation}\label{fii}
f_i(x(t),s(t))=\beta(s(t)) b_i(x(t),s(t)) (1 - x_i(t)) - \delta(s(t)) x_i(t),
\end{equation}
whereas $g$ is a diagonal matrix with entries   

\begin{equation}\label{gii}
g_{ii}(x(t),s(t)) =\sigma_i(x_i(t),s(t)) b_i(x(t),s(t)) (1-x_i(t)).
\end{equation}

Moreover, let us denote $\bR_+:=[0,\infty)$, and let $C^{2,1}(\bR^{N} \times \bR_+ \times S; \bR_+)$ be the family of all non-negative real-value functions $V(x,t,s)$ on $\bR^N \times \bR_+ \times S$, which are 
continuously twice differentiable in $x$ and once differentiable
in $t$. 
We define the differential operator $L$ on functions
 $V \in C^{2,1}(\bR^{N} \times \bR_+ \times S; \bR_+)$
 by setting
\begin{equation}\label{LV}
LV(x,t,s) = \frac{\partial V(x,t,s)}{\partial t}+ \sum_{i=1}^N  \frac{\partial V(x,t,s)}{\partial_{x_i}} f_i(x,s) + \frac{1}{2}\sum_{i=1}^N \frac{\partial^2 V(x,t,s)}{\partial x_i \partial x_i}  g_{ii}(x,s)^2 
+ \sum_{r=1}^m q_{sr} V(x,t,r).  
\end{equation}

In the following sections, we study the dynamic properties of the model described above. Precisely, in Section \ref{sec:extinction}, we show that there exists a positive global
solution for the system \eqref{ito}: the result is achieved by proving a somehow stronger property, namely that the solution never leaves the domain $(0,1)^N$, provided it starts inside it. Afterwards, we give a
sufficient condition for the almost sure extinction, and in Section \ref{sec:permanence} sufficient conditions for the stochastic permanence, meaning that the epidemics will survive indefinitely in the population with positive probability. We show that the mentioned conditions have closed relations with the stationary probability
distribution of the Markov chain \rev{and with the topology of the network}. In Section \ref{sec:invar}, we estimate the limit of time average 
of the sample path of the solution, \rev{providing the persistence in time mean of the system}. 
Then, based on this estimate, we prove the existence of an invariant probability measure of the {Markov process} $(x(t),s(t))$ on $(0,1)^N$. Under different condition, we prove the positive recurrence of $(x(t),s(t))$ and its ergodic properties in $(0,1)^N$. 
Most of the results in this paper are obtained through Lyapunov functions techniques, that were developed for stochastic
differential equations by Khasminskii \cite{khasminskii2011}, and later employed by many other authors (see, e.g., \cite{Arnold, Gard,mao2006stochastic}).
\rev{This kind of investigation 
has applications, e.g., in optimization and optimal control problems
\cite{zhu2007asymptotic,yin2009hybrid,ghosh1997ergodic}.}

\section{Global solution and extinction}\label{sec:extinction}

For a stochastic differential equation with Markovian switching, conditions assuring a unique global (i.e., no explosion in a finite
time) solution, for any given initial data, involve linear growth
 and local Lipschitz continuity of the coefficients of the equation \cite{mao2006stochastic}.
In our case, the coefficients of system \eqref{ito} are locally Lipschitz continuous, but they do not satisfy the linear growth condition, thus the solution may explode in a finite time. Thus, in the following theorem we ensure the no explosion in any finite time, by proving a somehow stronger property, that is $(0,1)^N$ is a positive invariant domain for \eqref{ito}.

\begin{theorem}\label{thm:01}
For any initial condition $x_0 \in (0,1)^N$ and $s_0 \in S$, and for any choice of system parameters $\beta(\cdot)$, $\delta(\cdot)$, any $A(\cdot)$ and $M(\cdot)$ in condition \eqref{cond}, there exists a unique global solution to system \eqref{ito}
on $t \geq 0$ and the solution remains in $(0,1)^N$ almost surely for all $t \geq 0$.
\end{theorem}

\begin{proof}
Let $0=\tau_0 < \tau_1 < \tau_2 < \ldots, < \tau_n < \ldots$ be the jump times of the Markov chain $s(t)$, and
let $s_0 \in S$ be the starting state. Thus $s(t)=s_0$ on $[\tau_0,\tau_1)$, and the subsystem for $t \in [\tau_0, \tau_1)$ has the following form:

\begin{equation*}
\begin{aligned}
 {\mathrm d} x_i(t) =& \left[\beta(s_0) b_i(x(t),s_0) (1 - x_i(t)) - \delta(s_0) x_i(t)\right] \, {\mathrm d}t
 + \sigma_i(x_i(t),s_0) b_i(x(t),s_0) (1 - x_i(t)) \, {\mathrm d}w_i(t), \qquad i \in \{1, \dots, N\}
\end{aligned}
\end{equation*}
By \cite[Theorem 1]{bonaccorsi2016epidemics}, the above subsystem has the solution $x(t) \in (0,1)^N$, on $t \in [\tau_0, \tau_1)$ and, by continuity for $t = \tau_1$, as well. Thus, $x(\tau_1) \in (0,1)^N$ and by considering $s(\tau_1)=s_1$, the subsystem for $t \in [\tau_1, \tau_2)$
becomes

\begin{equation*}
\begin{aligned}
 {\mathrm d} x_i(t) =& \left[\beta(s_1) b_i(x(t),s_1) (1 - x_i(t)) - \delta(s_1) x_i(t)\right] \, {\mathrm d}t
 + \sigma_i(x_i(t),s_1) b_i(x(t),s_1) (1 - x_i(t)) \, {\mathrm d}w_i(t), \quad i \in \{1, \dots, N\}
\end{aligned}
\end{equation*}
Again, by \cite[Theorem 1]{bonaccorsi2016epidemics}, the above subsystem has the solution $x(t) \in (0,1)^N$, for $t \in [\tau_1, \tau_2)$ and, by continuity for $t = \tau_2$, as well. Repeating this process continuously, we obtain that the solution $x(t)$ of system \eqref{ito} remains in $(0,1)^N$ with probability one, for all $t \geq 0$.
\end{proof}

\so{\begin{remark}\label{rem:boundary}
Let $\Delta=[0,1]^N$.
{From \eqref{ito}, we can easily see that if $x_0 \in \partial \Delta \setminus \left\{\bf 0 \right\}$, then $x(0+) \in (0,1)^N$.
Consequently, from Theorem \ref{thm:01}, for all $t > 0$ 
we have that $x(t) \in (0,1)^N$ a.s.}
\end{remark}}

In the sequel, we will need a few more notations. For any $x \in \mathbb{R}^N$, we denote by $|x|$ its Euclidean norm.
The symbols ${\bf 1}$ and ${\bf 0}$, denote the $N$-dimensional
column vector with all entries equal to 1 and 0, respectively. 
Moreover, let $G=(V,E)$ be an undirected graph and $A=(a_{ij})_{i,j=1, \ldots N}$ its adjacency matrix, then the degree of the $i$-th node, $d_i={\sum_{j=1}^N} a_{ij}$, is the number of its neighbors, and we denote by $d_{min}=\displaystyle \min_{0 \leq i \leq N} \{d_i\}$ the minimum degree.
In the following, we shall indicate the spectral radius of $A(s)$, the adjacency matrix corresponding to the state $s$, with $\lambda_1(s)$. Moreover, we shall denote with $d_{\min}(s)$ the minimum degree of the graph in state $s$.

Let us note that if $x_i >0$, $i=1, \ldots, N$, it holds 
$$\left(\sum_{i=1}^N x_i^2\right)^{\frac{1}{2}} \leq \sum_{i=1}^N x_i \leq \sqrt N \left(\sum_{i=1}^N x_i^2 \right)^{\frac{1}{2}}.$$
We will use such inequality several times through the paper.

\begin{remark}
For any $s \in S$, the adjacency matrix $A(s)$ is symmetric and
satisfies, for every $x \in \bR^N$,
\begin{equation}\label{cond1}
\langle A(s) x, x \rangle \le \lambda_1(s) |x|^2,
\end{equation}

\begin{equation}\label{cond2}
\langle A(s) x, A(s) x \rangle \le \lambda_1(s)^2 |x|^2.
\end{equation}
\rev{where $\langle \cdot, \cdot \rangle$ denotes the scalar product.}

\end{remark}

In the following theorem we find a sufficient condition, \rev{involving the model parameters and the network topology, together with the stationary distribution of the Markov chain}, that ensures that the epidemic will become extinct exponentially almost surely.  

\begin{theorem}\label{thm:ext}
For any initial condition $x_0 \in (0,1)^N$ and $s_0 \in S$, the solution $x(t)$ of system \eqref{ito} has the property

\begin{equation*}
\limsup_{t \to \infty} \frac{\log |x(t)|}{t} \leq \sum_{s=1}^m \pi_s \alpha(s)  \qquad \text{a.s.},
\end{equation*}
where 
\begin{equation*}
\alpha(s)= - \delta(s) + \beta(s) \lambda_1(s) + K(s),
\end{equation*}
with $K(s)= \frac{M(s)^2\lambda_1(s)^2} {32}$, $s \in S$.\\\\
In particular, if 
\begin{equation*}
    \sum_{s=1}^m \pi_s \alpha(s) < 0, 
\end{equation*}
then
\begin{equation}\label{ext}
\lim_{t \to \infty} |x(t)|=0 \qquad \text{a.s.}
\end{equation}
\end{theorem}

\begin{proof}
 Let us define $V: (0,1)^N \rightarrow \mathbb{R}_+$,
\begin{align*}
V(x(t))= \log|x(t)|.
\end{align*}
By the generalized It\^o formula (see, \cite[Sec.2.1]{MR1020057})
), and recalling \eqref{fii}, and \eqref{gii} we have:

\begin{align*}
{\mathrm d} V(x(t)) &= \sum_{i=1}^N \frac{x_i(t)}{|x(t)|^2}  \left(f_i(x(t),s(t))
{\mathrm d}t + 
g_{ii}(x(t),s(t)) {\mathrm d} w_i(t)\right) \\
 & + \frac{1}{2}\sum_{i=1}^N \left( \frac{1}{|x(t)|^2}   - \frac{ 2 x_i(t)^2}{|x(t)|^4} \right) g_{ii}(x(t),s(t))^2 {\mathrm d}t 
\end{align*}
from where we obtain
\begin{equation}\label{dVlog}
{\mathrm d} V(x(t))  \leq\sum_{i=1}^N \frac{x_i(t)}{|x(t)|^2}  \left(f_i(x(t),s(t)){\mathrm d}t + g_{ii}(x(t),s(t)) {\mathrm d} w_i(t) \right) 
 + \frac{1}{2}\sum_{i=1}^N  \frac{1}{|x(t)|^2}  g_{ii}(x(t),s(t))^2 {\mathrm d}t.
 \end{equation}
 
Now (here we omit for a while the explicit dependence on time of $x_i(t)$, $b_i(x(t),s(t))$ and $s(t)$ for readability) we have that
\begin{align}\label{dis1}
\sum_{i=1}^N \frac{x_i}{|x|^2} f_i(x,s) &= \sum_{i=1}^N \frac{x_i}{|x|^2}  \left[\left(\beta(s) b_i(x,s) (1-x_i) -\delta(s) x_i\right) \right]   \\ \nonumber
& =\frac{x^T}{|x|^2} \left[ \left(\beta(s) A(s) -\delta(s) I \right) x - \beta(s) \diag{(x)} A(s) x \right]  \\ \nonumber
&
\leq \frac{x^T}{|x|^2}  \left(\beta(s) A(s) -\delta(s) I \right) x  
 \leq  \beta(s) \lambda_1(s) -\delta(s) ,
\end{align}
where we have used condition \eqref{cond1}. Moreover, by Theorem \ref{thm:01} we know that $x_i \in (0,1)$ for all times, then it holds that
$x_i(1-x_i) \leq \frac14,$
 hence, by this and from \eqref{cond} and \eqref{cond2}, we have

 \begin{align}\label{dis2}
&\frac{1}{2}\sum_{i=1}^N  \frac{1}{|x|^2} \sigma^2_i(x_i,s) b^2_i(x,s) (1-x_i)^2 
\leq \frac{1}{32}M(s)^2  \lambda_1(s)^2
\end{align}
Thus, substituting inequalities \eqref{dis1} and \eqref{dis2} in \eqref{dVlog}, we have 

\begin{equation*}
{\mathrm d} V(x(t)) \leq \left(  \beta(s(t)) \lambda_1(s(t)) -\delta(s(t)) + \frac{1}{32}M(s(t))^2  \lambda_1(s(t))^2\right) {\mathrm d}t +
 \sum_{i=1}^N \frac{x_i(\tau)}{|x(\tau)|^2} g_{ii}(s(\tau), x_i(\tau)) \, {\rm d}w_i(\tau),
\end{equation*}
Thus, it holds
\begin{equation}\label{Vdis}
V(x(t)) \leq V(x(0)) + \int_0^t \alpha(s(u)) {\mathrm d} u + \tilde M(t),    
\end{equation}
where $\alpha(s)=\beta(s) \lambda_1(s) -\delta(s) + \frac{1}{32}M(s)^2  \lambda_1^2(A(s))$, and
$$\tilde M(t)=\sum_{i=1}^N \int_0^t \frac{x_i(u)}{|x(u)|^2} g_{ii}(x_i(u),s(u)) \, {\rm d}w_i(u).$$ 
\so{Let us note that $\tilde M(t) $ is a martingale.}
The adaptedness of the integrand function follows from the same property of $\{x_i\}$ and $\{s\}$.
It remains to prove that
\begin{align*}
    \bE \int_0^T \sum_{i=1}^N \left| \frac{x_i(u)}{|x(u)|^2} g_{ii}(x_i(u),s(u)) \right|^2 \, {\rm d}u < +\infty.
\end{align*}
From condition \eqref{cond} and since $b_i(s) \leq N-1$, 
we obtain
\begin{align*}\label{dis3}
\sum_{i=1}^N \left| \frac{x_i(u)}{|x(u)|^2} g_{ii}(x_i(u),s(u)) \right|^2 \leq 
|M(s(u))(N-1)|^2
\end{align*}
\so{which proves that $\tilde M(t)$ is a well defined stochastic integral on $[0,T]$ for arbitrary $T > 0$}.
Thus, by the strong law of large numbers for martingales (see \cite{mao2007book})

\begin{equation*}
  \lim_{t \to \infty} \frac{\tilde M(t)}{t}= 0  \qquad \text{a.s.}
\end{equation*}
Finally, from \eqref{Vdis}, by dividing $t$ on the both sides and then letting $t \to \infty$, we obtain, by the ergodic theorem of Markov chain (see e.g., \cite[Sec. 5.5]{resnick1992adventures}) 

\begin{equation*}
\limsup_{t \to \infty} \frac{V(x(t))}{t} \leq \limsup_{t \to \infty} \frac{1}{t} \int_0^t \alpha(s(u)) {\mathrm d} u = \sum_{s=1}^m \pi_s \alpha(s) \quad a.s.    
\end{equation*}
and, if $\sum_{s=1}^m \pi_s \alpha(s) < 0$, the assertion \eqref{ext} holds.
\end{proof}

\section{Stochastic permanence}\label{sec:permanence}

In this section, we discuss another relevant asymptotic behaviour of the solution of \eqref{ito}, the stochastic permanence, showing under which conditions it holds.

\begin{definition}
The system \eqref{ito} is said to be \emph{stochastically permanent} if for any $\varepsilon \in (0,1)$, there exists a positive constant
$\chi = \chi(\varepsilon)$ such that, for any initial condition $x_0 \in (0,1)^N$ and $s_0 \in S,$ the solution satisfies
\begin{equation*}
\label{e:201115.perm}
\liminf_{t \to \infty} \bP(|x(t)| \ge \chi) \ge 1-\varepsilon.
\end{equation*}
\end{definition}

\begin{assumption}\label{ass1}

For some $s \in S$, $q_{rs} >0$, for any $r \neq s$.
\end{assumption}

\begin{lemma}\label{lemma:Mmatrix}
Let Assumption \ref{ass1} holds. Let 
\begin{equation*}\label{baralpha}
\bar \alpha(s)= - \delta(s) + \beta(s) d_{\min}(s) - K(s),
\end{equation*}
where $K(s)= \frac{M(s)^2 \lambda_1(s)^2} {32}$, $s \in S$. 
Then, if $\sum_{s=1}^m \pi_s \bar{\alpha}(s) > 0$, there exists a \so{ sufficiently small positive constant $\theta$} 
such that the matrix 
\begin{equation}\label{Ttheta}
T(\theta)= \diag(\xi_1(\theta), \ldots, 
\xi_m(\theta)) - Q
\end{equation}
is a non singular $M$-matrix, where
$$\xi_s(\theta)=\theta \bar \alpha(s) -\theta^2 K(s), \qquad s \in S.$$
\end{lemma}
 
The proof follows from the same arguments in \cite[Lemma 3.4]{li2009population} and we omit it.

\begin{lemma}\label{lemma:supE}
If there exists a constant \so{$0<\theta <1$} such that $T(\theta)$ is a non singular $M$-matrix, then for any initial condition $x_0 \in (0,1)^N$, $s_0 \in S$, the solution $x(t)$ of system \eqref{ito} has the property that  
\begin{equation*}\label{supE}
\limsup_{t \to \infty} \bE \left[ \frac{1}{|x(t)|^\theta}\right] \leq H,
\end{equation*}
where $H$ is a positive constant.
\end{lemma}

\begin{proof}
Let us define the functions $V,U : (0,1)^N \rightarrow \mathbb{R}_+ $,
 \begin{equation}\label{VU}
 V(x(t))=\sum_{i=1}^N x_i(t),  \qquad U(x(t))=\frac{1}{V(x(t))}
 \end{equation}
 Then, by the generalized It\^o formula, we have
 \begin{equation}\label{dV}
 {\mathrm d}V(x(t))= \sum_{i=1}^N \left[  f_i(x(t),s(t)){\mathrm d}t +
g_{ii}(x(t),s(t)) {\mathrm d} w_i(t) \right],
\end{equation}
 and
\begin{align}\label{dU}
 {\mathrm d}U(x(t)) &= -U(x(t))^2{\mathrm d}V +U(x(t))^3 ({\mathrm d}V(x(t)))^2 \\ \nonumber
 &= -U(x(t))^2\sum_{i=1}^N   \left[ f_i(x(t),s(t)){\mathrm d}t + g_{ii}(x(t),s(t)) {\mathrm d} w_i(t) \right] + U(x(t))^3 \sum_{i=1}^N  g_{ii}(x(t),s(t))^2  {\mathrm d}t\\ \nonumber
 &=\left[-U(x(t))^2 \sum_{i=1}^N    f_i(x(t),s(t)) + U(x(t))^3 \sum_{i=1}^N g_{ii}(x(t),s(t))^2 \right]  {\mathrm d}t 
 - \sum_{i=1}^N U(x(t))^2 g_{ii}(x(t),s(t)) {\mathrm d} w_i(t)
 \end{align}
Now, by \cite[Theorem 2.10]{mao2006stochastic}, since 
$T(\theta)$ is a non singular $M$-matrix, 
there is a vector $z=(z_1, z_2, \ldots, z_m)^T \gg 0$  (meaning that all elements are positive) such that $T(\theta)z \gg 0,$
 that is
 \begin{equation*}
    z_s \theta \left( -\delta(s) + \beta(s) d_{\min}(s) - K(s) (1+\theta)\right)- \sum_{r=1}^m q_{sr} z_r >0, \qquad \text{for all} \quad 1\leq s \leq m.
 \end{equation*}
 with $K(s)=\frac{M(s)^2 \lambda_1(s)^2} {32}.$
  Hereafter, we drop the dependence on time of the processes $x_i(t)$, $b_i(x(t),s(t))$ and $s(t)$, for convenience.
 Define the function $\bar V: (0,1)^N \times S \rightarrow \bR_+$
 $$\bar V(x,s)=z_s U(x)^\theta.$$
  Applying \eqref{LV}, 
we have
\begin{align}\label{lbarV}
L \bar V(x,s) =& z_s \theta U(x)^{\theta-1} \left[ -U(x)^2 \sum_{i=1}^N   f_i(x,s) + U(x)^3 \sum_{i=1}^N g_{ii}(x,s)^2 \right] \\ \nonumber
& + \frac{1}{2} z_s \theta (\theta-1) U(x)^{\theta + 2} \sum_{i=1}^N g_{ii}(x,s)^2 + \sum_{r=1}^m q_{sr}z_r U(x)^\theta \\ \nonumber
=& U(x)^{\theta-1} \left[ z_s \theta\left( -U(x)^2 \sum_{i=1}^N   f_i(x,s) + U(x)^3 \sum_{i=1}^N  g_{ii}(x,s)^2\right) +\sum_{r=1}^m q_{sr}z_r U(x) \right]\\ \nonumber
& + \frac{1}{2} z_s \theta (\theta-1) U(x)^{\theta + 2} \sum_{i=1}^N g_{ii}(x,s)^2 
 \end{align}
 Now, let us consider
\begin{align}
\label{fibound}
-U(x)^2\sum_{i=1}^N  f_i(x,s) 
= U(x)^2 \left(-\sum_{i=1}^N \beta(s) b_i(x,s) + \sum_{i=1}^N \beta(s)  b_i(x,s) x_i + \sum_{i=1}^N \delta(s)  x_i \right).
\end{align}
Then, from \eqref{cond1}
\begin{align}\label{bAxx}
\beta(s) \sum_{i=1}^N  b_i(x,s) x_i & = \beta(s) \sum_{i,j=1}^N a_{ij}(s) x_j  x_i\\ \nonumber
& = \beta(s) \langle A(s) x, x \rangle \le 
\beta(s) \lambda_1(s) \, |x|^2 \le   \beta(s) \lambda_1(s) \, U^{-2}, 
\end{align}
 Moreover, 
\begin{align}\label{bAx}
-\sum_{i=1}^N \beta(s) b_i(x,s) + \sum_{i=1}^N \delta(s) x_i &=- \beta(s) \langle {\bf 1} , A(s) x \rangle + \delta(s) \langle {\bf 1} ,  x \rangle  = - \beta(s) \langle A^T(s) {\bf 1},  x \rangle  + \delta(s) \langle \bf{1},  x \rangle \\ \nonumber
 & \leq -\beta(s) d_{min}(s) \langle {\bf 1},  x \rangle + \delta(s) \langle {\bf 1} ,  x \rangle  =  \left(- \beta(s) \, d_{min}(s) + \delta(s)\right) U(x)^{-1}.
\end{align}
Using these estimates in \eqref{fibound}, we get
\begin{align}
\label{fibound2}
- U(x) ^2\sum_{i=1}^N f_i(x,s) 
\le \beta(s) \lambda_1(s)  + \left(- \beta(s) \, d_{\min}(s) + \delta(s)\right)U(x).
\end{align}
Next, we consider 
\begin{align*}
\sum_{i=1}^N  g_{ii}(x,s)^2 =\sum_{i=1}^N \left[ \sigma_i(x,s) b_i(x,s) (1-x_i) \right]^2;
\end{align*}
By the same arguments for \eqref{dis2}, we obtain that the
previous sum is upper bounded by
\begin{align}\label{gibound}
\frac{M^2}{16}  \, \sum_{i=1}^N   \left[   \sum_{j=1}^N a_{ij}(s) x_j \right]^2 &  = \frac{M^2}{16} \langle A(s)x, A(s)x \rangle 
 \leq \frac{M^2}{16} \lambda_1(s)^2 U(x)^{-2},
\end{align}
Thus, from \eqref{lbarV}, by the bounds in \eqref{fibound2} and \eqref{gibound}, we obtain
\begin{align}\label{lbarV2}
L \bar V(x,s) \leq& U(x)^{\theta-1} \left\{- \left[ z_s \theta\left(  \beta(s) d_{\min}(s) -\delta(s)  -\frac{M^2}{16} \lambda_1(s)^2 \right) - \sum_{r=1}^m q_{sr} z_r \right] U(x)  \right. \\ \nonumber
&+ z_s  \theta  \beta(s) \lambda_1(s)  \Bigg\}+  \frac{M^2}{32} \lambda_1(s)^2 z_s \theta (\theta-1) U(x)^\theta \\ \nonumber
=& U(x)^{\theta-1} \left\{- \left[ z_s \theta\left(  \beta(s) d_{\min}(s) -\delta(s)  -\frac{M^2}{32} \lambda_1(s)^2 (1+\theta) \right) - \sum_{r=1}^m q_{sr} z_r \right] U(x) \right. \\ \nonumber
&+ z_s  \theta  \beta(s) \lambda_1(s)  \Bigg\} = U(x)^{\theta-1}\left( - c_1(s) U(x) + c_0(s) \right)
\end{align}
 At this point, let us choose a constant $\kappa > 0$ sufficiently small such that $T(\theta) z-\kappa z \gg 0$, that is,
\begin{align}\label{condgg}
   z_s \theta \left( -\delta(s) + \beta(s) d_{\min}(s) - K(s)(1+\theta)\right)- \sum_{r=1}^m q_{sr} z_r - \kappa z_s &= c_1(s)-\kappa z_s >0,  \\ 
   & \text{for all} \quad 1\leq s \leq m. \nonumber
\end{align}
  and consider the process
$$Z(x,t,s)=e^{\kappa t} \bar V(x,s).$$
Then, again by the generalized It\^o formula, from \eqref{lbarV2}, we have 
\begin{align}\label{LZ}
   LZ(x,t,s) = \kappa e^{\kappa t}z_s U(x)^\theta +e^{\kappa t} L\bar V(x,s)
    \leq e^{\kappa t} U(x)^{\theta-1}\left[-(c_1(s)- \kappa z_s) U(x) + c_0(s)\right]\\ \nonumber
\end{align}
Noting that $U(x) \geq 1/N$, we can assert that for some positive and finite constant $H(s)$, $1\leq s \leq m$, it holds
$$U(x)^{\theta-1}\left[-(c_1(s)- \kappa z_s) U(x) + c_0(s)\right]\leq H(s) < +\infty, $$
since from \eqref{condgg} $c_1(s)-\kappa z_s>0$. 
 Now, denoting by $\hat z =\displaystyle \min_{s \in S} z_s$, and $ \bar H= \dfrac{\displaystyle \max_{s \in S}  H(s)}{\kappa \hat z}$, from \eqref{LZ}
\begin{equation*}
    LZ(x,t,s) \leq  e^{\kappa t} \kappa \hat z \bar H.
\end{equation*}
Then, from the generalized It\^o formula, we have that (see, e.g., \cite[Sec.2.1]{MR1020057})

$$\bE [Z(x(t),t, s(t))] \leq  \bE [Z(x(0),0,s(0))] + \bE\int_0^t e^{\kappa \tau} \kappa \hat z \bar H {\mathrm d} \tau  $$
  whence 
  $$\bE[\bar V(x(t),s(t))] \leq \frac{\bE [Z(x(0),0,s(0))]}{e^{\kappa t}} + \bar H \hat z, $$
 and since
 
$$\bE[ U(x)^\theta] 
\leq \frac{1}{\hat z} \bE[z_s U(x)^\theta] \leq \frac {\bE [Z(x(0),0,s(0))]}{\hat z \; e^{\kappa t}} + \bar H,  $$
we have
$$\limsup_{t \to \infty} \bE [U (x(t))^\theta] \leq \bar H.$$

 Next, since $U(x)^{-1} = \langle {\bf 1},x \rangle \le \sqrt{N} |x|$, we have
$$\limsup_{t \to \infty} \bE \left[\frac{1}{|x|^\theta}\right] \leq (\sqrt{N})^{\theta}\bar H= H.$$
 \end{proof}

\begin{theorem}\label{thm:perm}
Under Assumption \ref{ass1}, if $\sum_{s=1}^m \pi_s \bar \alpha(s) >0$ holds, then system \eqref{ito} is stochastically permanent. 
\end{theorem}

\begin{proof}
From Lemma \ref{lemma:Mmatrix} there exists a constant $0<\theta<1$ such that $T(\theta)$, defined as in \eqref{Ttheta}, is a non singular $M$-matrix. Then, from Lemma \ref{lemma:supE} and by means of a simple application of Markov's inequality (see, e.g., \cite[Theorem 5]{bonaccorsi2016epidemics}), we have the claim.
\end{proof}

\ste{\begin{remark}
Let us suppose that only the values of $\beta$ and $\delta$ switch depending of the Markov chain, while the network topology, of which $A$ is the adjacency matrix, remains fixed over time. Then, the sufficient conditions for the stochastic permanence of \eqref{ito} become Assumption \ref{ass1}
and $$\sum_{s=1}^m \pi_s( - \delta(s) + \beta(s) \lambda_1(A) - K(s)) >0.$$ Indeed, let $u$ be the Perron eigenvector of $A$, i.e., it is the eigenvector corresponding to the spectral radius $\lambda_1(A)$, and the unique one such that $u>0$, we require \so{$|u|_1=1$}. In the proof of Lemma \ref{lemma:supE}, let us choose $V(x)=\sum_{i=1}^N u_i x_i$. Then, instead of inequalities \eqref{bAxx},\eqref{bAx} and \eqref{gibound}, we have that
\begin{align*}
\beta(s) \sum_{i=1}^N u_i b_i(x,s) x_i & = \beta(s) \sum_{i,j=1}^N a_{ij} x_j u_i x_i \leq  \beta(s) \sum_{i,j=1}^N a_{ij} x_j  x_i\\
& = \beta(s) \langle A x, x \rangle \le 
\beta(s) \lambda_1(s) \, |x|^2 \le \bar u  \beta(s) \lambda_1(A) \, U^{-2}, 
\end{align*}
 since $u_i > 0$ and $|u|_1=1$, where $\bar u=\displaystyle\frac{1}{\displaystyle \min_i u_i^2}$, 
 \begin{align*}
-\sum_{i=1}^N\beta(s) b_i(x,s)  + \sum_{i=1}^N \delta(s) u_i x_i 
&=- \beta \langle u, A x \rangle + \delta(s) \langle u,  x \rangle  = - \beta \langle A^Tu,  x \rangle  + \delta(s) \langle u,  x \rangle \\
 & =  \left(- \beta(s) \, \lambda_1(a) + \delta(s)\right) U^{-1},
\end{align*}
and 
\begin{equation*}\label{gibound22}
\frac{M^2}{16}  \, \sum_{i=1}^N  u_i^2 \left[   \sum_{j=1}^N a_{ij}(s) x_j \right]^2   \le \frac{M^2}{16} \left[ \sum_{i=1}^N \left( u_i \sum_{j=1}^N  a_{ij}(s) x_j\right) \right]^2
= \frac{M^2}{16}
 \langle u, A x \rangle^2 = \frac{M^2}{16} \lambda_1(s)^2 U^{-2}.
\end{equation*}
Finally, with the same arguments in Lemmas \ref{lemma:Mmatrix}, \ref{lemma:supE}, and Theorem \ref{thm:perm}, we arrive to our assertion. The same considerations would apply if the network switches, but in such a way that $\lambda_1(s)$ remains the same in each state $s \in S$ (however quite unrealistic).
\end{remark}}

\section{Asymptotic bound of integral average and invariant probability measure}\label{sec:invar}

\rev{In this section, we estimate the time-average limit at infinity 
of the sample path of the solution. 
The lower bound found for this quantity is related to the stationary distribution of the Markov chain, the model parameters and the topology of the network. Basically, this result shows that system \eqref{ito}, pathwise, is \emph{persistent in the time average}.}

\begin{lemma}\label{lemma:zerotheta}
If there exists a constant $0<\theta<1$ such that $A(\theta)$ is a non singular $M$-matrix, then the solution $x(t)$ of \eqref{ito} with any initial values $x_0 \in (0,1)^N$ and $s_0 \in S$ has the following properties
\begin{equation}\label{limsupx}
\limsup_{t \to \infty} \frac{\log(|x(t)|)}{\log t} \leq 0, \qquad \text{a.s.}, 
\end{equation}
\begin{equation}\label{liminfx}
\liminf_{t \to \infty} \frac{\log(|x(t)|)}{\log t} \geq - \frac{1}{\theta}, \qquad \text{a.s.}
\end{equation}

\end{lemma}
\begin{proof}
Since $x(t)$ remains in $(0,1)^N$ for all positive times, it is straightforward to see that \eqref{limsupx} holds. 

Let us prove \eqref{liminfx}. Let $U : (0,1)^N \rightarrow \mathbb{R}_+ $ be defined as in \eqref{VU}, for simplicity we write $U(x(t))=U(t)$. By the generalized It\^o formula and \eqref{dU}, we have that
\begin{align}\label{diffe1U}
{\mathrm{d}}[ (1+U(t))^{\theta} ] &= \theta(1+U(t))^{\theta-1} {\mathrm{d}} U(t) \\ \nonumber
&+ \frac{1}{2} \theta(\theta-1)(1+U(t))^{\theta-2} U(t)^4 \sum_{i=1}^N g_{ii}(x(t),s(t))^2 {\mathrm{d}} t \\  \nonumber
&= \theta(1+U(t))^{\theta-1}\Bigg\{ \left(-U(x(t))^2 \sum_{i=1}^N    f_i(x(t),s(t)) + U(x(t))^3 \sum_{i=1}^N g_{ii}(x(t),s(t))^2 \right)  {\mathrm d}t \\ \nonumber
& - \sum_{i=1}^N U(x(t))^2 g_{ii}(x(t),s(t)) {\mathrm d} w_i(t) \Bigg\}
 + \frac{1}{2} \theta(\theta-1)(1+U(t))^{\theta-2} U(t)^4 \sum_{i=1}^N g_{ii}(x(t),s(t))^2 {\mathrm{d}} t \\ \nonumber
&= L(1+U(t))^\theta {\mathrm{d}}t - \theta (1+U(t))^{\theta-1} \sum_{i=1}^N g_{ii}(x(t),s(t)) {\mathrm{d}}w_i(t).
\end{align}
Now, by inequalities \eqref{fibound2} and \eqref{gibound}
\begin{align}\label{L1u}
L(1+U(t))^\theta  \leq& \theta(1+U)^{\theta-2}\Bigg\{- (1+U)\left(\beta(s(t)) d_{\min}(s(t)) -\delta(s(t)) -\beta(s(t)) \lambda_1(s) -\frac{M(s(t))^2}{16}\lambda_1(s(t))^2\right)U \nonumber
\\
&+ \frac{M(s(t))^2}{32}(\theta-1)  \lambda_1(s(t))^2 U^2\Bigg\} \\ 
\nonumber
 =& \theta(1+U(t))^{\theta-2} \Bigg\{- \left(\beta(s(t)) d_{\min}(s(t)) -\delta(s(t)) -\frac{M(s(t))^2}{32} \lambda_1(s(t))^2(1+\theta) \right)U(t)^2  \\ \nonumber
& \left. -\left(\beta(s(t)) d_{\min}(s(t)) - \beta(s(t)) \lambda_1(s(t)) -\delta(s(t)) -\frac{M(s(t))^2}{16}\lambda_1(s(t))^2   \right)U(t) \right. \Bigg\}.
\end{align}
Thus, by \eqref{L1u}, from \eqref{diffe1U}
\begin{align}\label{diffe1U2}
{\mathrm{d}}[ (1+U(t))^{\theta} ] 
 \leq& \theta(1+U(t))^{\theta-2} \Bigg\{- \left(\beta(s(t)) d_{\min}(s(t)) -\delta(s(t)) -\frac{M(s(t))^2}{32} \lambda_1(s(t))^2(1+\theta) \right)U(t)^2  \\ \nonumber
&  + K_1(1+U(t))  \Bigg\} - \theta (1+U(t))^{\theta-1} U(t)^2\sum_{i=1}^N g_{ii}(x(t),s(t)) {\mathrm{d}}w_i(t)
\end{align}
where $K_1=\max_{s \in S} \{c_1(s),c_0(s)\} $, 
where we define
$$c_1(s(t))= -\left(\beta(s(t)) d_{\min}(s(t)) - \beta(s(t)) \lambda_1(s(t)) -\delta(s(t)) -\frac{M(s(t))^2}{16}\lambda_1(s(t))^2   \right),$$ and $$c_0(s(t))=\beta(s(t))\lambda_1(s(t)),$$ 
With similar arguments as in Lemma \ref{lemma:supE}, one can prove that there exists $H>0$, such that
\begin{equation}\label{Esup2}
\bE[(1+U(t))^\theta] \leq H, \qquad \text{on} \quad t \geq 0, 
\end{equation}
Now, let $\eta >0$ and $k=1, 2,\ldots$, from \eqref{diffe1U2}, we obtain that
\begin{align}\label{EE}
 \bE &\left[  \sup_{(k-1)\eta \leq t \leq k\eta } (1+U(t))^\theta \right] \leq \bE\left[ (1+U((k-1)\eta))^\theta\right]  \\ \nonumber
  &+ \bE \left(\sup_{(k-1)\eta \leq t \leq k\eta } \left| \int_{(k-1)\eta}^t \theta(1+U(\tau))^{\theta-2}\Bigg\{ - \Bigg( \beta(s(\tau)) d_{\min}(s(\tau)) -\delta(s(\tau))  -\frac{M(s(\tau))^2}{32} \lambda_1(s(\tau))^2(1+\theta) \Bigg)U(\tau)^2\right.\right. \\\nonumber
 &+ K_1(1+U(\tau)) \Bigg\}  {\mathrm{d}} \tau \Bigg|  \Bigg) + \bE \left(\sup_{(k-1)\eta \leq t \leq k\eta }\left| \int_{(k-1)\eta}^t\theta (1+U(\tau))^{\theta-1} U(\tau)^2\sum_{i=1}^N g_{ii}(x(\tau),s(\tau)) {\mathrm{d}}w_i(\tau)\right| \right).
\end{align}
Next we can compute
\begin{align}\label{Efirst}
 \bE &\left(\sup_{(k-1)\eta \leq t \leq k\eta } \left| \int_{(k-1)\eta}^t \theta(1+U(\tau))^{\theta-2}\Bigg\{ - \Bigg( \beta(s(\tau)) d_{\min}(s(\tau)) -\delta(s(\tau)) -\frac{M(s(\tau))^2}{32} \lambda_1(s(\tau))^2(1+\theta) \Bigg)U(\tau)^2\right.\right. \\\nonumber
 &+ K_1(1+U(\tau)) \Bigg\}{\mathrm{d}} \tau \Bigg|  \Bigg) \\ \nonumber
 & \leq \bE  \left(  \int_{(k-1)\eta}^{k\eta} \left| \theta(1+U(\tau))^{\theta-2}\Bigg\{ - \Bigg( \beta(s(\tau)) d_{\min}(s(\tau)) -\delta(s(\tau))  -\frac{M(s(\tau))^2}{32} \lambda_1(s(\tau))^2(1+\theta) \Bigg)U(\tau)^2 \right. \right. \\ \nonumber
 & \left. \left. + K_1(1+U(\tau)) \Bigg\} \right| {\mathrm{d}} \tau   \right) \\ \nonumber 
 & \leq  \bE  \left(  \int_{(k-1)\eta}^{k\eta} \theta(1+U(\tau))^{\theta-2}\Bigg\{  \Bigg( \beta(s(\tau)) d_{\min}(s(\tau)) +\delta(s(\tau))  +\frac{M(s(\tau))^2}{32} \lambda_1(s(\tau))^2(1+\theta) + K_1 \Bigg)(1+U(\tau))^2  \Bigg\}  {\mathrm{d}} \tau   \right)\\ \nonumber
 & \leq \theta \eta K_2  \bE \left(\sup_{(k-1)\eta \, \leq t \leq k\eta } (1+U(t))^\theta  \right),
 \end{align}     
where $$K_2=\max_{s \in S}  \Bigg\{ \beta(s) d_{\min}(s) +\delta(s)  +\frac{M(s)^2}{32} \lambda_1(s)^2(1+\theta) + K_1 \Bigg\}.$$

By the Burkholder-Davis-Gundy inequality \cite[Lemma A.32]{yin2009hybrid}, there exists a positive constant $K_3$, such that
\begin{align}\label{Esecond}
\bE & \left(\sup_{(k-1)\eta \leq t \leq k\eta }\left|\int_{(k-1)\eta}^t\theta (1+U(\tau))^{\theta-1} U(\tau)^2\sum_{i=1}^N g_{ii}(x(\tau),s(\tau)) {\mathrm{d}}w_i(\tau)\right| \right) \\ \nonumber
& \so{\leq K_3 \bE \left(  \int_{(k-1)\eta}^{k\eta} (\theta (1+U(\tau))^{\theta-1} U(\tau)^2)^2 \sum_{i=1}^N g_{ii}(x(\tau),s(\tau))^2 {\mathrm{d}}\tau  \right)^{\frac{1}{2}}} \\ \nonumber
& \leq  K_3 \theta \bE \left(  \int_{(k-1)\eta}^{k\eta}(1+U(\tau))^{2 \theta} \frac{M^2(s(\tau))\lambda_1(s(\tau))^2}{16} {\mathrm{d}}\tau  \right)^{\frac{1}{2}} \\ \nonumber
&\leq K_3 \theta \max_{s \in S} \left\{\frac{M(s)}{4} \lambda_1(s)\right\} \bE \left(  \int_{(k-1)\eta}^{k\eta} (1+U(\tau))^{2\theta} {\mathrm{d}}\tau \right)^{\frac{1}{2}} \\ \nonumber
& \leq  K_3 \theta \max_{s \in S} \left\{\frac{M(s)}{4} \lambda_1(s)\right\} \eta^{\frac{1}{2}} \bE\left(\sup_{(k-1)\eta \, \leq t \leq k\eta } (1+U(t))^\theta \right),
\end{align}
where we have used the inequality \eqref{gibound}.
After substituting \eqref{Efirst} and \eqref{Esecond} into \eqref{EE}, we obtain
\begin{align*}
\bE \left[  \sup_{(k-1)\eta \leq t \leq k\eta } (1+U(t))^\theta \right] & \leq \bE\left[ (1+U((k-1)\eta))^\theta\right]  \\ \nonumber
& +\theta \left( K_2 \eta + K_3 \max_{s \in S} \left\{\frac{M(s)}{4} \lambda_1(s)\right\} \eta^{\frac{1}{2}} \right) 
\bE\left(\sup_{(k-1)\eta \, \leq t \leq k\eta } (1+U(t))^\theta \right).
\end{align*}
Now, letting $\eta>0$ sufficiently small to have
\begin{equation*}
\theta \left( K_2 \eta + K_3 \max_{s \in S} \left\{\frac{M(s)}{4} \lambda_1(s)\right\} \eta^{\frac{1}{2}} \right) < \frac{1}{2},
\end{equation*}
and by considering \eqref{Esup2}, it holds that
\begin{equation*}
    \bE\left(\sup_{(k-1)\eta \, \leq t \leq k\eta } (1+U(t))^\theta \right) \leq 2 H.
\end{equation*}
Let $\varepsilon >0$, by applying the Markov's inequality, we have
\begin{equation*}
\bP \left\{ \omega: \sup_{(k-1)\eta \, \leq t \leq k\eta} (1+U(t))^\theta > (k\eta)^{1+\varepsilon}  \right\} \leq \frac{2H}{(k \eta)^{1+ \varepsilon}}  \qquad k=1,2 \ldots
\end{equation*}
Applying the Borel-Cantelli lemma (see e.g.  \cite{mao2007book}), for almost all $\omega \in \Omega$
\begin{equation}\label{supU}
    \sup_{(k-1)\eta \, \leq t \leq k\eta} (1+U(t))^\theta \leq (k\eta)^{1+\varepsilon}
\end{equation}
holds for all but finitely many $k$.
Thus, there exists an integer $k_0(\omega) > 1/\eta+2$ for almost all $\omega \in \Omega$, for which \eqref{supU} holds for every $k \geq k_0$. 
Then, for almost all $\omega \in \Omega$, if $k \geq k_0$ and $(k-1)\eta \leq t \leq k\eta$,
\begin{equation*}
\frac{\log(1+U(t))^\theta}{\log t} \leq \frac{(1+\varepsilon)\log(k\eta)}{\log((k-1)\eta)} \to 1+ \varepsilon   \quad \text{as} \quad k  \quad \text{increases}.
\end{equation*}
Letting $\varepsilon \to 0$ and taking the upper limit, we have
\begin{equation*}
    \limsup_{t \to \infty} \frac{\log(1+U(t))^\theta}{\log t} \leq 1 \quad \text{a.s.}
\end{equation*}
Thus, remembering the definition of $U(t)$, we have
\begin{equation*}
\liminf_{t \to \infty}\frac{\log(\sum_{i=1}^N x_i(t))}{\log t} \geq - \frac{1}{\theta} \quad \text{a.s.}   
\end{equation*}
From which, since $\sqrt N |x| \geq \sum_{i=1}^N x_i $ we finally arrive to
\begin{equation*}
\liminf_{t \to \infty}\frac{\log |x(t)|}{\log t} \geq - \frac{1}{\theta} \quad \text{a.s.}   
\end{equation*} \end{proof}

\begin{theorem}\label{thm:comp}
 Under Assumption \ref{ass1}, and condition $\sum_{s=1}^m \pi_s \bar \alpha(s) >0$, for any initial values $x_0 \in (0,1)^N$ and $s_0 \in S$, the solution $x(t)$ of \eqref{ito} obeys
 \begin{equation*}
    \liminf_{t \to \infty} \frac{1}{t} \int_0^t |x(u)| {\mathrm{d}} u \geq \frac{1}{\hat K} \sum_{s=1}^m \pi_s \bar \alpha(s) \quad \text{a.s.}
 \end{equation*}
 where $\hat K= \displaystyle \max_{s \in S} \beta(s)\lambda_1(s)$.
\end{theorem}
\begin{proof}
Let us consider $V : (0,1)^N \rightarrow \mathbb{R}_+ $ as in \eqref{VU} and its stochastic differential \eqref{dV}.
From the assertions in Lemma \ref{lemma:zerotheta} it is immediate to observe that 
\begin{equation}\label{limV0}
 \lim_{t \to \infty}\frac{\log(V(x(t)))}{t}= 0, \qquad \text{a.s.} 
\end{equation}  
By using \eqref{dV} we compute
\begin{equation}\label{dlogV}
{\mathrm{d}}\log V(x(t))= 
 \frac{1}{V(x(t))} \sum_{i=1}^N \Big[  f_i(x(t),s(t)){\mathrm d}t +
g_{ii}(x(t),s(t)) {\mathrm d} w_i(t) \Big] - \frac{1}{2V(x(t))^2} \sum_{i=1}^N g_{ii}(x(t),s(t))^2 {\mathrm{d}} t
\end{equation}
Now, let us drop for a while $x(t)$ from $V(x(t))$ and omit the explicit dependence on time of $x_i(t)$, $b_i(t)$ and $s(t)$ for readability. Then, by using inequalities \eqref{bAxx} and \eqref{bAx}, we have 
\begin{align*}
 \frac{1}{V} \sum_{i=1}^N  f_i(x,s)&= \frac{1}{V} \Big[\sum_{i=1}^N  \beta(s) \sum_{j=1}^N a_{ij}(s)x_j - \sum_{i=1}^N \beta(s) x_i\sum_{j=1}^N a_{ij}(s) x_j- \sum_{i=1}^N \delta(s)  x_i \Big] \\ \nonumber
  &\geq \frac{1}{V} \Big( \beta(s) d_{\min}(s) \sum_{i=1}^N x_i - \beta(s) \lambda_1(s) |x| \sum_{i=1}^N   x_i - \delta(s) \sum_{i=1}^N   x_i \Big) \\ \nonumber
  & \geq \beta(s) d_{\min}(s) -\hat K |x| - \delta(s),
 \end{align*}
where $\hat K= \displaystyle \max_{s \in S} \beta(s)\lambda_1(s)$.
Moreover, by inequality \eqref{gibound}, we have
\begin{equation*}
    - \frac{1}{2V^2} \sum_{i=1}^N g_{ii}(x,s)^2 \geq -\frac{M(s)^2}{32}\lambda^2_1(A(s)) 
\end{equation*}
Substituting the above estimates in \eqref{dlogV} gives
\begin{align*}
  {\mathrm{d}}\log V(x(t)) &\geq \Big(\beta(s(t)) d_{\min}(s(t))  - \delta(s(t)) -\frac{M^2(s(t))}{32}\lambda^2_1(A(s(t)))\Big){\mathrm{d}}t -\hat K |x(t)| {\mathrm{d}}t \\ \nonumber
  & + \frac{1}{V(x(t))} \sum_{i=1}^N g_{ii}(x(t),s(t)) {\mathrm d} w_i(t)
\end{align*}
Whence,
\begin{equation}\label{logV}
\log V(x(t))+ \hat K \int_0^t |x(u)| {\mathrm{d}}u \geq \log V(x(0)) + \int_0^t \bar \alpha(s(u)) {\mathrm d}u + \int_0^t \sum_{i=1}^N\frac{g_{ii}(x(u),s(u))}{V(x(u))}  {\mathrm d} w_i(u)
\end{equation}
By the strong law of large numbers for martingales (see \cite{mao2007book}), we have
\begin{equation*}
\lim_{t \to \infty} \frac{1}{t} \int_{0}^t  \sum_{i=1}^N \frac{g_{ii}(x(u),s(u))}{V(x(u))}  {\mathrm d} w_i(u)=0 \quad \text{a.s.}  
\end{equation*}
Therefore, dividing both sides of \eqref{logV} by $t$ and then letting $t\ \to \infty$, {by \eqref{limV0} and the Ergodic theorem for Markov chain} \cite[Sec. 5.5]{resnick1992adventures}, we have 
\begin{equation*}
\liminf_{t \to \infty} \frac{1}{t}\int_0^t |x(u)| {\mathrm{d}}u \geq \frac{1}{\hat K} \sum_{s=1}^m \pi_s \bar \alpha(s) \quad \text{a.s.} 
\end{equation*}
\end{proof}

\subsection{Existence of an invariant probability measure}

\rev{The presence of a stochastic perturbation destroys the existence positive deterministic equilibria. 
However, in the analysis of a stochastic dynamical system, one can try to prove the existence (and possibly the uniqueness) of an invariant probability measure (stationary distribution) on a positive invariant domain to better understand the long-term behavior of the system. }

\rev{In this section and in the following one, we investigate the existence of a stationary distribution and its ergodic property for the 
two-component process} $\left\{ (x(t),s(t)), t \geq 0\right\}$.\\
Consider the positive invariant set $$ \cH=(0,1)^N \times S$$ for the system \eqref{ito} (see Theorem \ref{thm:01}).
Let us note that $(x(t),s(t))$ is a \so{time-homogeneous Markov process}, and \so{since the hypothesis of \cite[Lemma 2.14]{yin2009hybrid} are satisfied in $[0,1]^N$, then it follows that our process has the Feller property (see \cite[Remark 2.15]{yin2009hybrid}).}\\\\


We use the following theorem to ensure the existence of an invariant probability measure on $\cH$.

\begin{theorem}[see \cite{stettner1986existence}]\label{existmu}
Let $\Phi=\left\{ \Phi_t, t \geq 0 \right\}$ be a Feller process with state space $(X,\cB(X))$.  Then either

\begin{itemize}
\item[a)] there exists an invariant probability measure on $X$, or
\item[b)] for any compact set $C \subset X$,
$$\lim_{t \to \infty}\sup_{\mu} \frac{1}{t} \int_0^t \left( \int_X \bP(u,x,C) \mu( {\mathrm d}x)\right){\mathrm d} u=0,  $$
where the supremum is taken over all initial distributions $\mu$ on the state space $X$, and $\bP(t,x,C)=\bP_x(\Phi_t \in C)$ is the transition probability function.
\end{itemize}
\end{theorem}

\begin{theorem}\label{exiinv}
Under Assumption \ref{ass1}, and condition $\sum_{s=1}^m \pi_s \bar \alpha(s) >0$, the Markov process $z(t)=(x(t),s(t))$ has an invariant probability measure $\mu^*$ on the state space $\cH$.
\end{theorem}

\begin{proof}
Let us consider the process $(x(t),s(t))$ on a larger state space 
$$\tilde \cH = \Delta \setminus \left\{\bf 0 \right\} \times S,$$
where, we recall, $\Delta=[0,1]^N$ and $\bf 0$ is the zero-vector in $\mathbb R^N$. By means of Theorem \ref{existmu}, one can prove the existence of an invariant probability measure $\mu^*$ for the process $z(t)=(x(t),s(t))$ on $\tilde \cH$, providing the existence of a compact set $C \subset \tilde \cH$ that verifies

\begin{equation}\label{liminfP}
\liminf_{t \to \infty}  \frac{1}{t} \int_0^t \left( \int_{\tilde \cH} \bP(u,z,C) \mu( {\mathrm d}z)\right){\mathrm d} u = \liminf_{t \to \infty}  \frac{1}{t} \int_0^t \bP(u,z_0,C) {\mathrm d}u > 0, 
\end{equation}
for some initial distribution $\mu=\delta_{z_0}$, with $z_0 \in \tilde \cH$, where $\delta$ is the Dirac function and $\bP(u,\cdot,\cdot)$ is the transition probability function. Once the existence of $\mu^*$ is proved, we can easily see that $\mu^*$ is also the invariant probability measure of $(x(t),s(t))$ on $\cH$. \so{Indeed, from Theorem \ref{thm:01} and Remark \ref{rem:boundary}) there exists a time $\bar t >0$, such that for all $t \geq \bar t$, $(x(t),s(t)) \in \cH$, therefore $\mu^*\left( \partial \Delta \times S \right)=0$.}

Thus, it is sufficient to find a compact set $C \subset \tilde \cH$ that satisfies \eqref{liminfP} to prove the existence of an invariant probability measure for the process $(x(t),s(t))$ on $\cH$. \\ By Theorem \ref{thm:comp}, we have
\begin{equation*}
\liminf_{t \to \infty} \frac{1}{t} \int_0^t |x(\tau)| {\mathrm{d}} \tau \geq \frac{1}{\hat K} \sum_{s=1}^m \pi_s \bar \alpha(s) := \zeta>0 \quad \text{a.s.}
 \end{equation*}
 where $\hat K= \displaystyle\max_{s \in S} \beta(s)\lambda_1(s)$. 
 Since
 \begin{equation*}
     \frac{1}{t}\int_0^t |x(u)| {\mathrm d}u = \frac{1}{t} \int_0^t |x(u)| \mathbbm{1}_{\left\{|x(u)| < \frac{\zeta}{2}\right\}}{\mathrm d}u + \frac{1}{t} \int_0^t |x(u)| \mathbbm{1}_{\left\{|x(u)| \geq \frac{\zeta}{2}\right\}}{\mathrm d}u 
     \leq \frac{\zeta}{2} + \sqrt N  \frac{1}{t}  \int_0^t  \mathbbm{1}_{\left\{|x(u)| \geq \frac{\zeta}{2}\right\}}{\mathrm d}u,
 \end{equation*}
we get
\begin{equation}\label{dislim}
    \liminf_{t \to \infty} \frac{1}{t} \int_0^t \mathbbm{1}_{\left\{|x(u)| \geq \frac{\zeta}{2}\right\}}{\mathrm d}u \geq \frac{\zeta}{2 \sqrt N} \qquad \text{a.s.}
\end{equation}
By Fatou's lemma, and from \eqref{dislim}, it follows that
\begin{multline*}
\liminf_{t \to \infty} \frac{1}{t} \int_0^t \bP\Big(|x(u)| \geq \frac{\zeta}{2} \Big) {\mathrm d}u =  
\liminf_{t \to \infty} \frac{1}{t} \int_0^t \bE\left[ \mathbbm{1}_{\left\{|x(u)| \geq \frac{\zeta}{2}\right\}}\right] {\mathrm d}u \\ \nonumber
 \geq \bE \left[ \liminf_{t \to \infty} \frac{1}{t} \int_0^t \mathbbm{1}_{\left\{|x(u)| \geq \frac{\zeta}{2}\right\}} {\mathrm d}u  \right] 
\geq \frac{\zeta}{2\sqrt {N}}.
\end{multline*}
Now, let us consider the compact set $C=D \times S \subset \tilde \cH$, where 
$D=\left\{ x \in \mathbb{R}^N: 0 \leq x_i \leq 1, |x| \geq \frac{\zeta}{2} \right\}$. Then,
\begin{align*}
 \liminf_{t \to \infty}  \frac{1}{t} \int_0^t \bP(u,z_0,C) {\mathrm d}u =  \liminf_{t \to \infty} \frac{1}{t} \int_0^t \bP\Big(|x(u)| \geq \frac{\zeta}{2} \Big) {\mathrm d}u \geq \frac{\zeta}{2 \sqrt N}
\end{align*}
which implies that \eqref{liminfP} holds.
\end{proof}

\subsection{Positive recurrence and ergodicity}

In the previous section, we proved the existence of an invariant probability measure under conditions ensuring the stochastic permanence. In this section, under a different condition, we are able to prove the ergodicity of $z(t)=(x(t),s(t))$ in $(0,1)^N$. According to \cite{zhu2007asymptotic}, it suffices to prove that $z(t)$ is positive recurrent (see \cite[Definition 2.3]{zhu2007asymptotic}). {From \cite{zhu2007asymptotic}, we can deduce the following lemma that gives a condition for the positive recurrence of $z(t)=(x(t),s(t))$ with respect to some domain $D \times \{s\}$, with $\bar D \subset (0,1)^N$, and $s \in S$ (and consequently, with respect to any nonempty open subset of $(0,1)^N \times S$). From this the ergodicity of the process in $ (0,1)^N$  follows.}

\begin{lemma}\label{lemma:lyap}
The Markov process $z(t)=(x(t),s(t))$ is positive recurrent if the following conditions are satisfied:

\begin{itemize}
    \item [i)] \so{for each $s \in S$}, and for each open set $D$ such that $\bar D\subset (0,1)^N$, there exists $k_D>0$ s.t.
    \begin{equation*}
        k_D |\xi|^2\leq \xi^T g(x,s) \xi \leq k_D^{-1}|\xi|^2 \qquad \text{for all} \quad \xi \in \mathbb{R}^N, 
    \end{equation*}
    with some constant $k_D \in (0,1]$ for all $x \in D$,
     \item [ii)] there exists an open subset $D$ with a sufficiently smooth boundary and $\bar D \subset (0,1)^N$, such that for each $s \in S$, there exists a  twice continuously differentiable nonnegative function $V(\cdot,s): D^c \cap (0,1)^N \to \mathbb{R}$ such that for some $\alpha > 0$,
     
     $$LV(x,s) \leq -\alpha, \qquad (x,s) \in D^c \cap (0,1)^N \times S.$$
     
\end{itemize}

Moreover, the positive recurrent Markov process $z(t)=(x(t),s(t))$ has a unique stationary distribution $\nu(\cdot,\cdot)$ on $\cH$, and it holds that

\begin{equation}\label{erg}
    \mathbb P_{x,i} \left(\lim_{T \to \infty} \frac{1}{T} \int_0^T f(x(t),\gamma(t)) dt =\sum_{i=1}^m \int_{(0,1)^N} f(x,i) \nu (dx,i)\right)=1,
\end{equation}
where $f:(0,1)^N\times S \to \mathbb{R}$ is any
integrable function with respect to the measure $\nu(\cdot,\cdot)$.
\end{lemma}

\begin{remark}
Let us note that in \cite{zhu2007asymptotic} the authors assume that the differential operator $L$ is uniformly elliptic (assumption (A), condition (i) in \cite{zhu2007asymptotic}). In our case, the uniform ellipticity is not satisfied in $(0,1)^N$, however we can replace it by the weaker condition (i) in Lemma \ref{lemma:lyap}.
It is straightforward to see that their results still hold.
{Let us also note that it is sufficient to replace the condition ii) of the assumption (A) in \cite{zhu2007asymptotic}, with our request of irreducibility of the matrix $Q$}. 
\end{remark}

\begin{theorem}
Let $\tilde \alpha(s)=\beta(s) d_{min}(s)- N \delta -N\frac{M(s)^2 \lambda_1(A)^2}{2}$. If $\sum_{s=1}^m\pi_s \tilde \alpha(s) >0$, the process $(x(t),s(t))$ has a unique stationary distribution $\nu$ on $\cH$, with the ergodic property \eqref{erg}.
\end{theorem}

\begin{proof}

 Let us choose $0<\eps<1$, and consider the following bounded open subset:

$$D_{\eps}=\left(\eps, 1-\eps\right)^N \subset(0,1)^N.$$
Then $\bar D_\eps \subset(0,1)^N$. 
 Since the diffusion matrix $g(x,s)$ is a diagonal matrix with entries $\sigma^2_i(x_i(s))(1-x_i(s))^2b_i(x,s)^2$, the condition i) in Lemma \ref{lemma:lyap} is easily verified. 
 
 We need to verify condition ii).
Since the matrix $Q$ is irreducible, there exists a solution $\omega=(\omega_1, \ldots, \omega_m)^T$ of the Poisson system 

\begin{equation}\label{eq:Qw}
 \tilde \alpha-Q \omega=(\pi^T \tilde \alpha, \ldots, \pi^T \tilde \alpha)^T,   
\end{equation}
where $\tilde \alpha=(\tilde \alpha(1), \ldots, \tilde \alpha(m))^T$ \cite[Lemma 2.3]{khasminskii2007stability}. Then, let $\bar \omega$ be a constant such that $\omega_s+\bar \omega>0$ for all $s \in S$, and consider the function
$V(\cdot,s): D^c \cap (0,1)^N \to \mathbb{R}$,

\begin{equation*}
V(x,s)= - k_1  \sum_{i=1}^N \log(x_i) - \sum_{i=1}^N \log(1-x_i) + k_1(\omega_s+ \bar \omega),  \qquad s=1, \ldots, m.  
\end{equation*}
where $k_1>0$. Denote $V_1=- k_1  \sum_{i=1}^N \log(x_i)$, and $V_2=- \sum_{i=1}^N \log(1-x_i) $. By the generalized It\^o formula

\begin{align*}
  LV_1 &= -k_1 \sum_{i=1}^N\frac{\beta(s)(1-x_i)b_i(x,s)-\delta x_i}{x_i} + \frac{k_1}{2}  \sum_{i=1}^N \frac{\sigma_i(x,s)^2b_i(x,s)^2(1-x_i)^2}{x_i^2}\\
   &  \leq - k_1\left[\frac{\sum_{i=1}^N \beta(s)(1-x_i)b_i(x,s)}{x_i}- \delta N  +\frac{1}{2} M(s)^2 \lambda_1(s)^2 N \right]
\end{align*}

\begin{align*}
 LV_2 &= \sum_{i=1}^N\frac{\beta(s)(1-x_i)b_i(x,s)-\delta x_i}{1-x_i} + \frac{1}{2} \sum_{i=1}^N \frac{\sigma_i(x,s)^2b_i(x,s)^2(1-x_i)^2}{(1-x_i)^2}\\
 & \leq  \beta(s)\lambda_1(s) N- \sum_{i=1}^N \frac{\delta x_i}{1-x_i} + \frac{1}{2} M(s)^2 \lambda_1(A)^2 N 
\end{align*}

From which

\begin{align}\label{LVerg}
LV(x,s) &\leq - k_1\left[\frac{\sum_{i=1}^N \beta(s)(1-x_i)b_i(x,s)}{x_i}- \delta N  +\frac{1}{2} M(s)^2 \lambda_1(s)^2 N \right] \\    
&+ \beta(s)\lambda_1(s) N- \sum_{i=1}^N \frac{\delta x_i}{1-x_i} + \frac{1}{2} M(s)^2 \lambda_1(A)^2 N  + k_1 \sum_{r=1}^m q_{sr} \nonumber \omega_r
\end{align}

At this point, we can have two cases: 

\emph{Case 1}: there exists $j$ such that $x_j > 1- \eps$. Since
$$\frac{-\delta x_j}{(1-x_j)} \to -\infty, \quad \text{as} \quad x_j \to 1,$$
we have from \eqref{LVerg} that, for a sufficient small $\eps$,  $LV(x,s) \leq - \alpha$, for some $\alpha >0$.

\emph{Case 2}: Suppose that for all $i$, $x_i < 1- \eps$, then there exist $j$ such that $x_j < \eps$.
Without loss of generality we assume that $x_j=\displaystyle \min_{i} x_i$. 
Then 
\begin{align}\label{LVerg2}
    LV(x,s) &\leq -k_1 \left[ \beta(s)(1-x_j)\frac{\sum_{j=1}^m a_{jl}x_l}{x_j} -\delta N-\frac{1}{2} M(s)^2 \lambda_1(A)^2 N-\sum_{r=1}^m q_{sr}\omega_r  \right]\\ \nonumber
    & + \beta(s)\lambda_1(s) N+ \frac{1}{2} M(s)^2 \lambda_1(A)^2 N\\ \nonumber
   & \leq -k_1 \left[ (1-\eps) \beta(s) d_{min}(s)-\delta N-\frac{ M(s)^2 \lambda_1(A)^2 N}{2} -\sum_{r=1}^m q_{sr}\omega_r \right]\\
   & + \beta(s)\lambda_1(s) N+ \frac{1}{2} M(s)^2 \lambda_1(A)^2 N. \nonumber
\end{align}
Since $\sum_{s=1}^m \pi_s \tilde \alpha(s)>0 $, from \eqref{eq:Qw}, we have
$$\beta(s) d_{min}(s)-\delta N-\frac{ M(s)^2 \lambda_1(A)^2 N}{2}-\sum_{r=1}^m q_{sr}\omega_r>0.$$
Now, we can choose a constant $\eps >0$ sufficiently small such that
$$(1-\eps)\beta(s) d_{min}(s)-\delta N-\frac{ M(s)^2 \lambda_1(A)^2 N}{2}-\sum_{r=1}^m q_{sr}\omega_r>0,$$
and $k_1>0$ big enough to have, from \eqref{LVerg2}, 
$LV(x,s) \leq -\alpha$, for some $\alpha>0$. Thus, also condition ii) in Lemma \ref{lemma:lyap} is verified for $D_\eps$, and we can conclude that $(x(t),s(t))$ is positive recurrent, has a unique stationary distribution $\nu$ on $\cH$, and \eqref{erg} holds.
\end{proof}


\section{Numerical experiments}

In this section, we provide some numerical investigations of the obtained results, based on our theoretical assumptions.  In all numerical experiments in the following, we consider a right-continuous Markov chain with state space $S=\{1,2\}$, thus we study the random switching of \eqref{ito} between two subsystems. In particular, we analyze the case where in the subsystem corresponding to $s=1$ the epidemic goes extinct a.s., while the subsystem in $s=2$ is stochastically permanent. Let us note that, from \cite[Theorem 5]{bonaccorsi2016epidemics}, a sufficient condition ensuring stochastic permanence of the subsystem in state $s$ is
\begin{equation*}\label{condlambda1}
- \delta(s) + \beta(s) \lambda_1(s) - \frac{M(s)^2 \lambda_1(s)^2} {32}>0.
\end{equation*}
Clearly, if $\bar \alpha(s) >0$, the above condition holds.\\
In Fig.~\ref{fig:reg} a) and b),
we depict the behaviour of the sample path solution norm of \eqref{ito} (simulated by the Euler–Maruyama method \cite{higham2001}). We consider the random switching of the system between two different sets of parameters, and two regular networks with $N=100$ nodes, one with degree $d=10$, when the system lies in the state $s=1$, and the other with $d=25$, in $s=2$.  
  In the same plot, we also report the sample path of the Markov chain $s(t)$.
  \\
Specifically, in Fig.~\ref{fig:reg} a), we consider $q_{12}=0.2$, and $q_{21}=0.7$, consequently $\pi=(\pi_1, \pi_2)=\left(\frac{7}{9}, \frac{2}{9}\right)$. The parameter values
of system \eqref{ito}, in $s=1$, are
\begin{equation*}
\beta(1)=0.01, \qquad \delta(1)=1, \qquad M(1)=0.1, 
\end{equation*}
 while in $s=2$ 
\begin{equation*}
\beta(2)=0.07, \qquad \delta(2)=1, \qquad M(2)=0.05,
\end{equation*}
Thus, 
$$\sum_{s=1}^N \pi_s \alpha(s)=-0.4982 <0, 
$$
and by Theorem \ref{thm:ext}, as the result of Markovian switching, the epidemics will go extinct almost surely, in the long run. \\
In Fig.~\ref{fig:reg} b), instead, we consider $q_{12}=1$, and $q_{21}=0.3$, consequently $\pi=(\pi_1, \pi_2)=\left(\frac{3}{13}, \frac{10}{13} \right)$. The parameter values
of system \eqref{ito}, in $s=1$, are 
\begin{equation*}
\beta(1)=0.09, \qquad \delta(1)=1, \qquad M(1)=0.15, 
\end{equation*}
while in $s=2$  
\begin{equation*}
\beta(2)=0.1, \qquad \delta(2)=1, \qquad M(2)=0.08.
\end{equation*}
Thus, 
$$\sum_{s=1}^N \pi_s \bar \alpha(s)= 1.0184>0,$$
and by Theorem \ref{thm:perm}, as the result of Markovian switching, the overall behavior of system \eqref{ito} is stochastically permanent. 

 
\begin{figure}[h!]
 \centering
   \begin{minipage}{0.45\textwidth}
    \includegraphics[width=\textwidth,height=5cm]{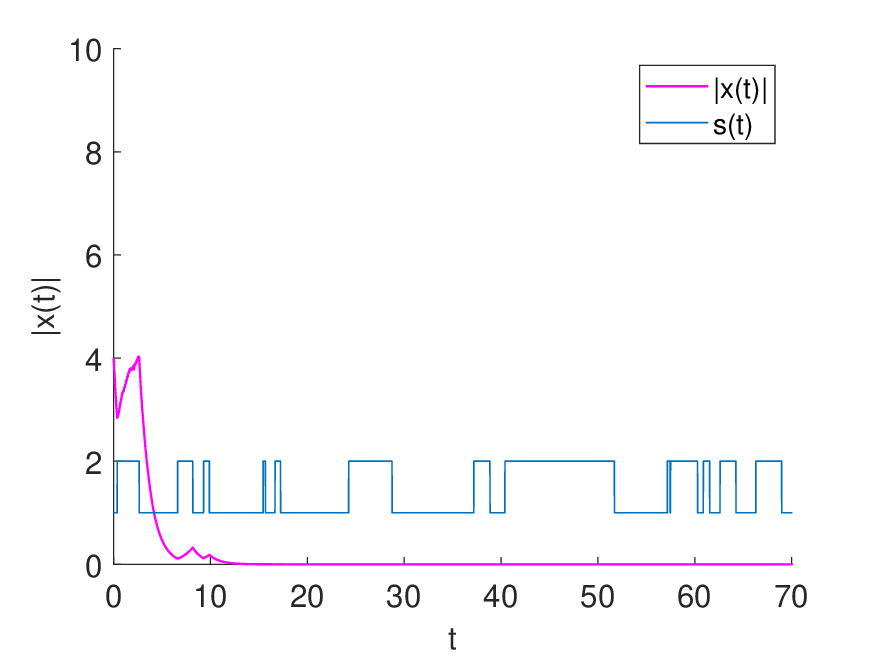}\put(-210,130){a)}
\end{minipage}
 \hskip4mm
   \begin{minipage}{0.45\textwidth}
    \includegraphics[width=\textwidth,height=5cm]{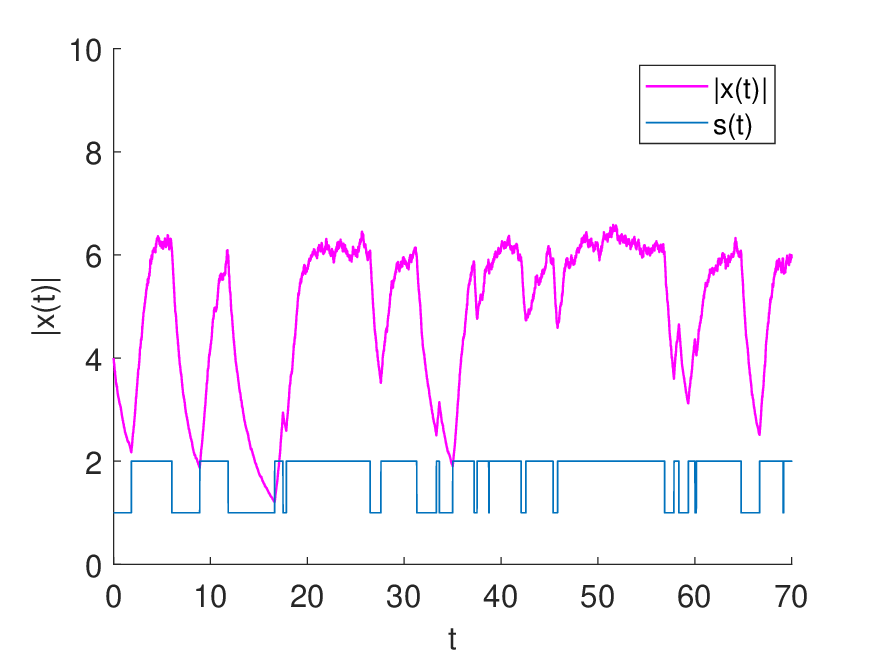}\put(-210,130){b)}
\end{minipage}
    \caption{Dynamics of the sample path solution norm of \eqref{ito} and the Markov chain $s(t)$, with $x(0)=0.4\cdot{\bf 1}$. Regular network with $N=100$ and degree $d=10$ in $s=1$, and with degree $d=25$, in $s=2$.  a) $s(0)=1$, $q_{12}=0.2$, $q_{21}=0.7$ b) $s(0)=1$, $ q_{12}=1$, $q_{21}=0.3$.}
    \label{fig:reg}
  \end{figure}

As a second scenario, we consider the random switching of \eqref{ito} between two \ER sample graphs with $N=100$ and $p=0.7, 0.2$, respectively.
\\
In Fig.~\ref{fig:erdos_ext} a) we
consider $q_{12}=0.4$, $q_{21}=0.8$, consequently $\pi=(\pi_1, \pi_2)=\left(\frac{2}{3}, \frac{1}{3}\right)$. The parameter values
of system \eqref{ito}, in $s=1$, are :
\begin{equation*}
\beta(1)=0.01, \qquad \delta(1)=1, \qquad M(1)=0.02, 
\end{equation*}
while in $s=2$  
\begin{equation*}
\beta(2)=0.065, \qquad \delta(1)=1, \qquad M(2)=0.1,
\end{equation*}
Thus, 
$$\sum_{s=1}^N \pi_s \alpha(s)=-0.0042<0,
$$
and by Theorem \ref{thm:ext}, as the result of Markovian switching, the epidemics will go extinct almost surely, in the long run. 

In Fig.~\ref{fig:erdos_perm}, we consider the switching between two \ER sample graphs with $N=100$ and $p=0.7, 0.4$, respectively.
In a), we have $q_{12}=0.1$, and $q_{21}=0.15$, consequently $\pi=(\pi_1, \pi_2)=\left(\frac{3}{5},\frac{2}{5} \right)$. The parameter values
of system \eqref{ito} in $s=1$ are 
\begin{equation*}
\beta(1)=0.01, \qquad \delta(1)=0.85, \qquad M(1)=0.03, 
\end{equation*}
while in $s=2$
\begin{equation*}
\beta(2)=0.06, \qquad \delta(2)=0.85, \qquad M(2)=0.04.
\end{equation*}
Thus, 
$$\sum_{s=1}^N \pi_s \bar \alpha(s)=0.08>0,$$
and by Theorem \ref{thm:perm}, as the result of Markovian switching, the overall behavior of system is stochastically permanent.
\ste{In Fig.~\ref{fig:erdos_perm} b), instead, we consider $q_{12}=1$, and $q_{21}=1.5$, which are proportional to those in Fig.~\ref{fig:erdos_perm} a),
 and the same parameter values of a). Let us note that this leads to have the same values of $\pi$, and $\sum_{s=1}^N \pi_s \bar \alpha(s)$, as in Fig.~\ref{fig:erdos_perm} a), thus the overall behavior of the system is stochastically permanent.
{From the comparison between Fig.~\ref{fig:erdos_perm} a) and b), we can see how the mean sojourn times in each state $s$, given by $1/q_{ss}$, influences the overall behaviour of the system}. Indeed, in Fig.~\ref{fig:erdos_perm} a) it is evident that the system can stay longer in state $s=1$, and the infection has the time to slow down and to reach a low epidemic level, before the switching to state $s=2$ triggers another aggressive epidemic wave. }
 
\begin{figure}[ht]
 \centering
\includegraphics[width=0.45\textwidth]{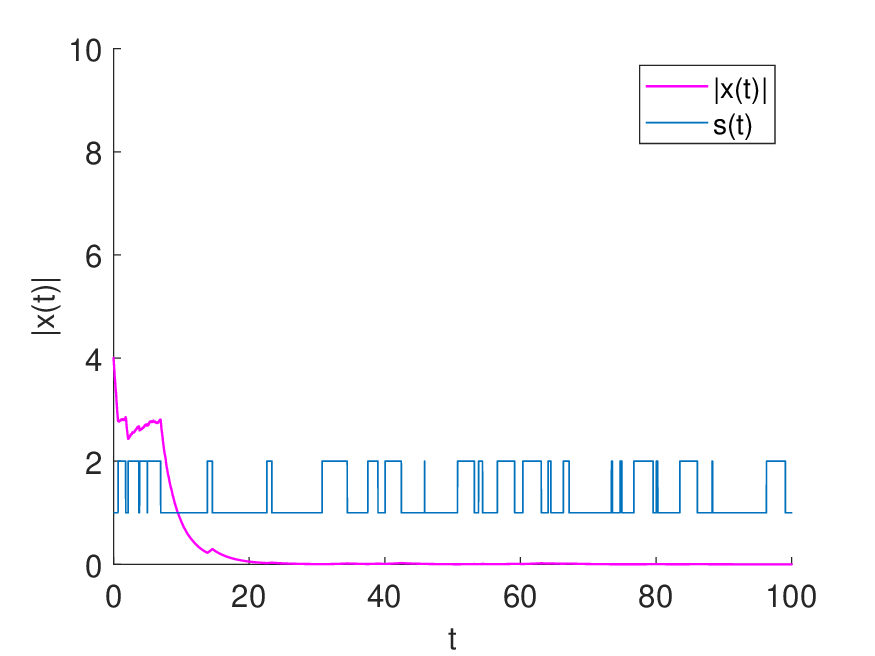}
\caption{Dynamics of the sample path solution norm of \eqref{ito}, and the Markov chain $s(t)$, with $x(0)=0.4\cdot{\bf 1}$, and $s(0)=1$, $q_{12}=0.4$, $q_{21}=0.8$. \ER sample graphs with $N=100$ and $p=0.7$ in $s=1$, and $p=0.2$ in $s=2$. }
    \label{fig:erdos_ext}
   \end{figure}
  
 \begin{figure}[ht]
 \centering
   \begin{minipage}{0.45\textwidth}
    \includegraphics[width=\textwidth,height=5cm]{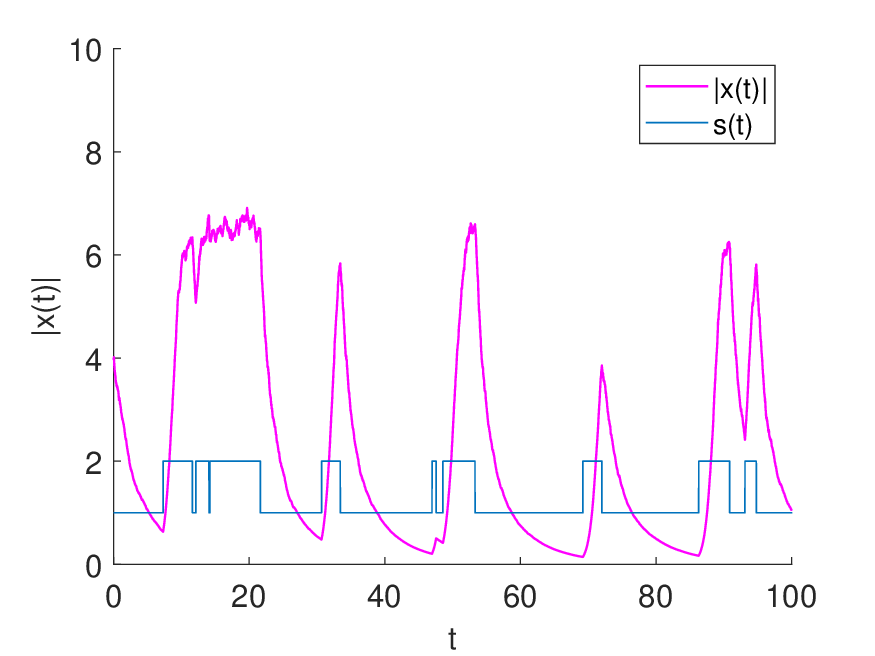}\put(-210,130){a)}
\end{minipage}
 \hskip4mm
   \begin{minipage}{0.45\textwidth}
    \includegraphics[width=\textwidth,height=5cm]{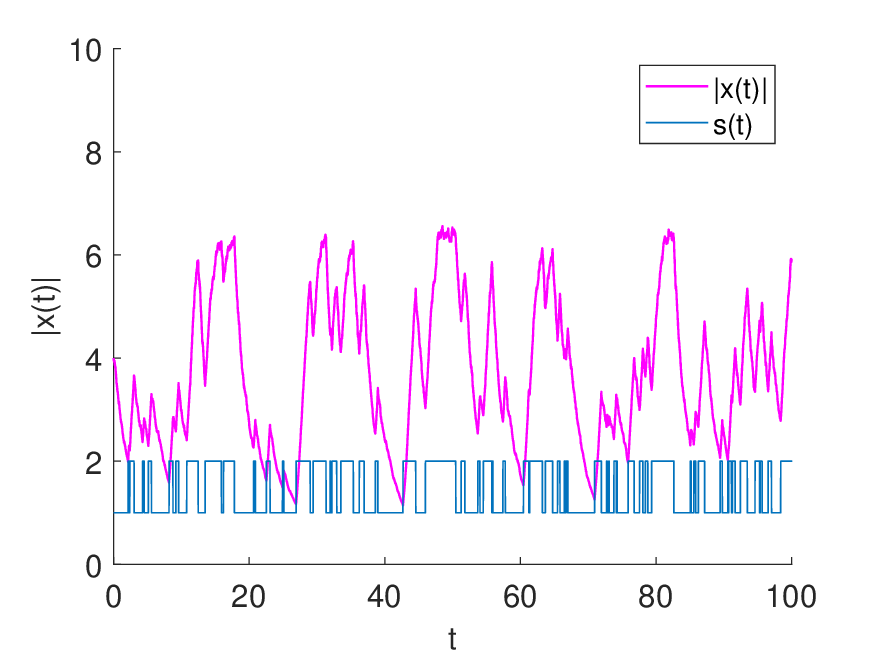}\put(-210,130){b)}
\end{minipage}
    \caption{Dynamics of the sample path solution norm of \eqref{ito} and the Markov chain $s(t)$, with $x(0)=0.4\cdot{\bf 1}$. \ER sample graphs with $N=100$ and $p=0.7$ in $s=1$, and $p=0.4$ in $s=2$. a) $s(0)=1$, $q_{12}=0.1$, $q_{21}=0.15$ b) $s(0)=1$, $ q_{12}=1$, $q_{21}=1.5$.}
    \label{fig:erdos_perm}
   \end{figure}

 \begin{figure}[ht]
 \centering
\includegraphics[width=0.45\textwidth]{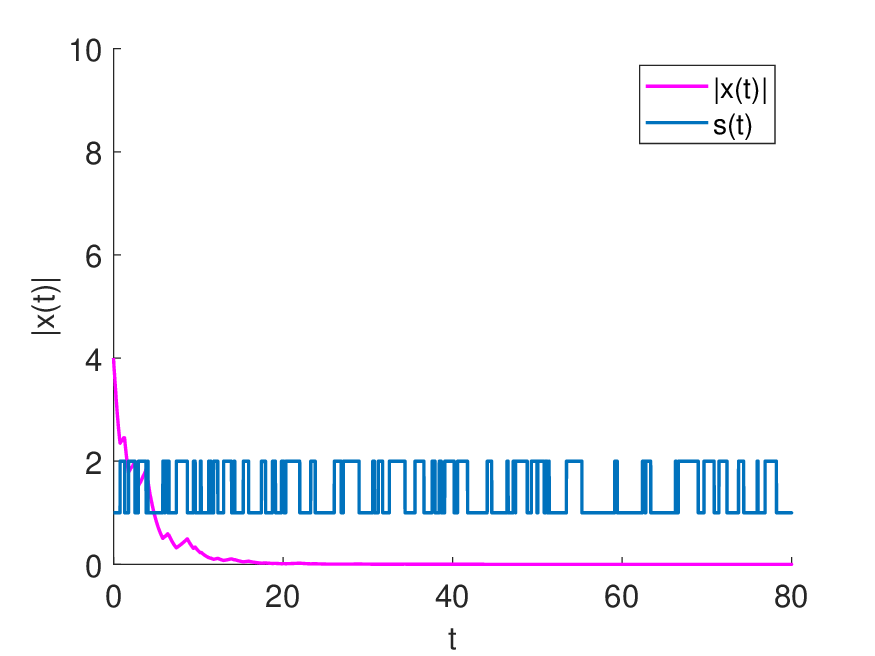}
\caption{Dynamics of the sample path solution norm of \eqref{ito}, and the Markov chain $s(t)$, with $x(0)=0.4\cdot{\bf 1}$, and $s(0)=1$, $q_{12}=1$, $q_{21}=1.5$. \ER sample graphs with $N=100$ and $p=0.2$ in $s=1$ and $s=2$. }
    \label{fig:erdos_ext2}
   \end{figure}  
   
 \rev{ In Fig.~\ref{fig:erdos_ext2}, we consider the same parameters and same transition rates as in Fig.~\ref{fig:erdos_perm} b), but we have for both states an \ER sample graph with $N=100$ and $p=0.2$.
  In this case, in each state the condition for extinction is satisfied and, consequently, the epidemic will go extinct almost surely, indeed 
 $$ \sum_{s=1}^N \pi_s \alpha(s)=-0.2135<0.$$
 Thus, we see that by keeping the model parameters fixed and changing only the network topology (with respect to Fig.~\ref{fig:erdos_perm} b)), the overall system long-term behavior changes, passing from a situation of stochastic permanence to the a.s. extinction. This clearly supports the actual fact that deleting some links between individuals (e.g, through lockdown or self-quarantines) can help to stop an epidemic.}
 

\section{Conclusion}
 In this paper, we have investigated a random-switching diffusion SIS model on networks, incorporating also heterogeneity in the transmission of the infection. Specifically, we have studied the dynamic behaviour of the hybrid system \eqref{ito}, that is composed of $m$ subsystems, and switches between them according to the law of the Markov chain.
 In this way, we have extended our study in \cite{bonaccorsi2016epidemics}. 
 
 In Theorem \ref{thm:01}, we have proved that the system \eqref{ito} possesses a unique global solution that remains within $(0, 1)^N$ a.s., whenever it starts from this region. 
 Theorems \ref{thm:ext} and \ref{thm:perm} provide sufficient conditions 
 for the almost sure extinction and the stochastic permanence of \eqref{ito}, respectively.
 Moreover, we can see that if $\alpha(s)=  - \delta(s) + \beta(s) \lambda_1(s) + \frac{M(s)^2 \lambda_1(s)^2} {32} <0$, for all $s \in S$, in each subsystem the epidemic goes extinct, and from Theorem \ref{thm:ext}, as we can expect, in the overall system \eqref{ito} the epidemic goes extinct. On the other hand, if $\bar \alpha(s)= - \delta(s) + \beta(s) d_{\min}(s) -  \frac{M(s)^2 \lambda_1(s)^2} {32}>0$ for all $s \in S$, each subsystem is stochastically permanent, and from Theorem \ref{thm:perm}, the overall system remains stochastically permanent. We remember that, from \cite[Theorem 5]{bonaccorsi2016epidemics}, a sufficient condition for stochastic permanence of each subsystem is $ - \delta(s) + \beta(s) \lambda_1(s) -  \frac{M(s)^2 \lambda_1(s)^2} {32}>0$. 
 However, Theorems \ref{thm:ext} and \ref{thm:perm} show us a more interesting behaviour, that is when some subsystems are stochastically permanent,
while in others a stochastic extinction is predict, as the results of Markovian switching, the overall epidemic can persist or go extinct a.s., which depends on the sign of $\sum_{s=1}^m \pi_s \bar \alpha(s)$ and $\sum_{s=1}^m \pi_s  \alpha(s)$, respectively. 
\rev{We can see that the 
stationary distribution of the Markov chain plays a
crucial role in the stochastic epidemic process. Thus, reasonably, the overall system will tend to the extinction a.s. if the process has higher probability to stay in regimes that predict the extinction. Moreover, in each state, one can control the epidemic by lowering the infection rate (e.g., by wearing face-masks) and/or reducing the spectral radius of the adjacency matrix (by deleting some links in the network, through e.g., lockdown or self-quarantines). Thus, it is evident also the relevant role of the network topology, the change of which can possibly help to reduce the epidemic, even if the other parameters remain fixed.} 

In Theorem \ref{thm:comp}, \so{an asymptotic bound for the time average
of the solution sample path norm has been provided, under the conditions ensuring stochastic permanence. Moreover, the obtained bound allowed us to prove the existence of an invariant probability measure for the process $(x(t),s(t))$ on $(0,1)^N$, if the condition of stochastic permanence holds.} \ste{ Under different condition, we have proved the positive recurrence of $(x(t),s(t))$ and its ergodic properties in $(0,1)^N$. Finally, we have corroborated the obtained results by numerical examples}. \\
 \rev{As future research, it would be interesting to consider switching processes that
depend on the continuous dynamics. Although more difficult to analytically handle,
this approach allows to represent, for example, the interplay between
node state and network dynamics. Indeed, in human disease epidemics, the structure of contacts
often changes in response to the contagion, which in turn influences the spreading process itself in a
nontrivial feedback loop.}
    
\section*{Acknowledgments}  
No potential competing interest was reported by the authors.\\
\rev{The authors would like to thank the referees and the Associate Editor for their helpful comments and suggestions}.


\providecommand{\bysame}{\leavevmode\hbox to3em{\hrulefill}\thinspace}
\providecommand{\MR}{\relax\ifhmode\unskip\space\fi MR }
\providecommand{\MRhref}[2]{%
  \href{http://www.ams.org/mathscinet-getitem?mr=#1}{#2}
}
\providecommand{\href}[2]{#2}

\end{document}